\documentclass{amsart}
\usepackage{amssymb}
\usepackage{mathrsfs}
\usepackage{txfonts}
\usepackage{amsfonts}
\usepackage{amstext}
\usepackage{amssymb}
\usepackage{amsmath}
\usepackage{graphicx}
\usepackage{hyperref}
\usepackage{url}
\usepackage{tikz}
\usepackage[final]{showkeys}
\usepackage{float}
\usepackage{boondox-cal}

\usetikzlibrary{patterns}
\hypersetup{
        colorlinks   = true,
        citecolor    = blue,
        linkcolor    = blue,
        urlcolor     = blue
}
\newtheorem{theorem}{Theorem}[section]
\newtheorem{proposition}[theorem]{Proposition}
\newtheorem{lemma}[theorem]{Lemma}
\newtheorem{corollary}[theorem]{Corollary}

\theoremstyle{definition}
\newtheorem{remark}[theorem]{Remark}

\newtheorem{example}[theorem]{Example}
\newtheorem{definition}[theorem]{Definition}

\numberwithin{equation}{section}

\def \color#1{}
\allowdisplaybreaks

\def\tbigcap{\mathop{\textstyle \bigcap }}

\def\tprod{\mathop{\textstyle \prod }}
\numberwithin{equation}{section}
\newcounter{counterConstant}

\makeatletter

\newcommand{\Rmnum}[1]{\expandafter\@slowromancap\romannumeral #1@}
\@namedef{subjclassname@2020}{\textup{2020} Mathematics Subject Classification}
\makeatother

\begin{document}
\title{A dichotomy on the self-similarity of graph-directed attractors}
\author[Falconer]{Kenneth J. Falconer}
\address{Mathematical Institute, University of St Andrews, North Haugh, St
Andrews, Fife, KY16 9SS, Scotland, UK.}
\email{kjf@st-andrews.ac.uk}
\author[Hu]{Jiaxin Hu}
\address{The Department of Mathematical Sciences, Tsinghua University,
Beijing 100084, China.}
\email{hujiaxin@tsinghua.edu.cn}
\author[Zhang]{Junda Zhang}
\address{The Department of Mathematical Sciences, Tsinghua University,
Beijing 100084, P.R. China.}
\email{zjd18@mails.tsinghua.edu.cn}
\thanks{\noindent Supported by the National Natural Science Foundation of
China (No. 11871296).}
\date{\today }

\begin{abstract}
This paper seeks conditions that ensure that the attractor of a graph directed
iterated function system (GD-IFS) cannot be realised as the attractor of a
standard iterated function system (IFS). For a strongly connected directed
graph, it is known that, if all directed
circuits go through a vertex, then for any GD-IFS of similarities on $\mathbb{R}$
based on the graph and satisfying the convex open set condition (COSC),
its attractor associated with this vertex is also the attractor of a (COSC)
standard IFS. In this paper we show the following complementary result.
If a directed circuit does not go through a vertex, then there exists a
GD-IFS based on
the graph such that the attractor associated with this vertex is not the
attractor of any standard IFS of similarities. Indeed, we give algebraic conditions for
such GD-IFS attractors not to be attractors of standard IFSs, and thus
show that `almost-all' COSC GD-IFSs based on the
graph have attractors associated with this vertex that are not the
attractors of any COSC standard IFS.
\end{abstract}

\maketitle
\tableofcontents

\section{Introduction}

An \emph{iterated function system} (IFS) $\{S_{i}\}_{i}$ is a finite set of
distinct contracting maps on a complete metric space which we will assume
here to be $\mathbb{R}^{n}$ \cite{Hutchison1981}. The \emph{attractor} of
the IFS is the unique nonempty compact set $K\subset \mathbb{R}^{n}$ such
that
\begin{equation}
K=\mathop{\textstyle \bigcup }\limits_{i=1}^{m}S_{i}(K).  \label{attract}
\end{equation}%
If these maps are all contracting similarities, we say that this IFS is a
\emph{standard IFS}, and call $K$ a \textit{self-similar set}. A contracting
similarity $S(x)$ on $\mathbb{R}$ can be written as $S(x)=\rho x+b$ where $%
\rho \in (-1,1)\setminus \{0\}$ is the \emph{contraction ratio}.

Separation conditions for IFSs are often required to ensure `not too much
overlapping' in the union (\ref{attract}). A frequent condition is the \emph{%
open set condition} (OSC), meaning that there exists a nonempty open set $%
U\subseteq \mathbb{R}^{n}$, such that $\mathop{\textstyle \bigcup }%
\limits_{i=1}^{m}S_{i}(U)\subseteq U$ with this union disjoint. We say that
the IFS satisfies the \emph{convex open set condition} (COSC) if $U$ can be
chosen to be convex, or we can (equivalently) take $U=\mathrm{int(conv}\,K)$
where `\textrm{conv}\thinspace ' denotes the convex hull, `\textrm{int}'
denotes the interior of a set. We say that the IFS satisfies the \emph{%
convex strong separation condition} (CSSC) if we can take $U=\mathrm{int(conv%
}\,K)$ such that $S_{i}(\mathrm{conv}\,K)\cap S_{j}(\mathrm{conv}%
\,K)=\emptyset $ for any $i\neq j$.

We also consider graph-directed IFSs \cite{MauldinWil1988} based on a given
digraph. A \emph{directed graph}\textit{\ }(or\textit{\ a digraph }for
brevity), $G:=\left( V,E\right) ,$ consists of a finite set of vertices $V$
and a finite set of directed edges $E$ (for brevity we often omit
`directed') with loops and multiple edges allowed. Let $E_{uv}\subset E$ be
the set of edges from the \textit{initial} vertex $u$ to the \textit{terminal%
} vertex $v$. A \textit{graph-directed iterated function system} (GD-IFS) on
$\mathbb{R}^{n}$ consists of a finite collection of contracting similarities
$\left\{ S_{e}:e\in E_{uv}\right\} $ from $\mathbb{R}_{v}^{n}$ to $\mathbb{R}%
_{u}^{n}$ for $u,$ $v\in V$, where $\mathbb{R}_{u}^{n}$ is a copy of $%
\mathbb{R}^{n}$ associated with vertex $u$. We write $\rho _{e}\in
(-1,1)\setminus \{0\}$ for the contraction ratio of the similarity $S_{e}$
in $%
%TCIMACRO{\U{211d} }%
%BeginExpansion
\mathbb{R}
%EndExpansion
$. We always require the digraph satisfies that $d_{u}\geq 1$ for every $%
u\in V$ (\cite{MauldinWil1988}, \cite[Section 4.3]{Edgar}), where $d_{u}$ is
the \emph{out-degree} of $u$ (the number of directed edges leaving $u$). For
a GD-IFS $\left( V,E,\left( S_{e}\right) _{e\in E}\right) $ based on such a
digraph, there exists a unique list of non-empty compact sets $(F_{u}\subset
\mathbb{R}_{u}^{n})_{u\in V}$ such that, for all $u\in V,$
\begin{equation}
F_{u}=\mathop{\textstyle \bigcup }\limits_{v\in V}\mathop{\textstyle \bigcup
}\limits_{e\in E_{uv}}S_{e}(F_{v}),  \label{gdattract}
\end{equation}%
see \cite{MauldinWil1988} or \cite[Theorem 4.3.5 on p.128]{Edgar}. We call
the above $\left( F_{u}\right) _{u\in V}$ the \emph{(list of) attractors }of
the GD-IFS, and each $F_{u}$ is called a \emph{GD-attractor}. A (finite)
\emph{directed path} $e_{1}e_{2}\mathbb{\cdots }e_{k}$ is a consecutive
sequence of directed edges $e_{i}\in E$ $(i=1,\mathbb{\cdots },k)$ for which
the terminal vertex of $e_{i}$ is the initial vertex of $e_{i+1}$ $(i=1,%
\mathbb{\cdots },k-1)$. For a directed path $\mathbf{e}=e_{1}e_{2}\cdots
e_{k}$ with edges $e_{i}\,(1\leq i\leq k)$, the corresponding contractive
mapping is given by $S_{\mathbf{e}}=S_{e_{1}}\circ S_{e_{2}}\circ \cdots
\circ S_{e_{k}}$, and its contraction ratio along $\mathbf{e}$ is $\rho _{%
\mathbf{e}}=\rho _{e_{1}}\rho _{e_{2}}\cdots \rho _{e_{k}}$.

For a GD-IFS there are analogous separation conditions. The \textit{open set
condition} (OSC) is satisfied if there exist non-empty bounded open sets $%
\left( U_{u}\subset \mathbb{R}_{u}^{n}\right) _{u\in V},$ with
\begin{equation}
\mathop{\textstyle \bigcup }_{v\in V}\mathop{\textstyle \bigcup }_{e\in
E_{uv}}S_{e}\left( U_{v}\right) \subset U_{u}  \label{COSC}
\end{equation}%
and the union is disjoint for each $u\in V.$ The \textit{convex open set
condition} (COSC) means that these $\left( U_{u}\right) _{u\in V}$ can all
be chosen to be convex. In one-dimensional case, one can take
\begin{equation}
(U_{u})_{u\in V}=(\mathrm{int}(\mathrm{conv}\,F_{u}))_{u\in V},  \label{con3}
\end{equation}%
since $\mathrm{conv}\,F_{u}\subset \overline{U_{u}}$ for each $u\in V$ (see
Proposition \ref{APP2} in the Appendix). We say that a GD-IFS satisfies the CSSC
(\emph{convex strong separation condition}), if the union
\begin{equation}
\mathop{\textstyle \bigcup }_{v\in V}\mathop{\textstyle \bigcup }_{e\in
E_{uv}}S_{e}\left( \mathrm{conv}\,F_{v}\right) \text{ \ (which belongs to }%
\mathrm{conv}\,F_{u}\text{)}  \label{CSSC}
\end{equation}%
is disjoint for each $u\in V$.

GD-attractors and GD-IFSs appear naturally in dynamical systems and fractal
geometry. For example, certain complex dynamical systems can be regarded as
conformal GD-IFSs using a Markov partition, see \cite[Section 5.5]%
{FalconerTechniques}. For another occurrence, the orthogonal projection of
certain self-similar sets may be GD-attractors \cite[Theorem 1.1]{Farkas2016}%
. We will work with COSC (including CSSC) GD-IFSs defined on $\mathbb{R}$
based on digraphs with $d_{u}\geq 2$ for every vertex $u$ in $V$ throughout
this paper.

We say that a digraph is strongly connected if, for all vertices $u,v\in V$,
there is a \emph{directed path} from $u$ to $v$ (we allow $u=v$). For
brevity, we will assume throughout that a strongly connected digraph always
satisfies $d_{u}\geq 2$ for all $u\in V$. This is because, if $d_{v}=1$ $%
(v\in V)$ then $F_{v}$ is just a scaled copy of another GD-attractor $F_{w}$
$(w\in V\setminus \{v\})$. Then $F_{v}$ is self-similar (with the COSC) if
and only if $F_{w}$ is self-similar (with the COSC), since if $K$ is the
attractor of the IFS $\{\rho _{i}x+b_{i}\}_{i}$, then $\eta K+l$ is the
attractor of the IFS $\{\rho _{i}x+\eta b_{i}+(1-\rho _{i})l\}_{i}$ $(\eta
,l\in \mathbb{R})$. We can do a reduction as in \cite[pp.607]{EdgarMau1992}
on any strongly connected digraph and associated GD-IFS, to obtain a
subgraph and new GD-IFS with $d_{u}\geq 2$ for all $u\in V$ such that each
attractor is similar to one of the original ones.

A natural question arises, \textquotedblleft When does a
GD-IFS of similarity mappings have attractors which cannot be realised
as attractors of any
standard IFS?\textquotedblright . In particular we seek algebraic conditions
involving the parameters underlying the GD-IFS similarities that ensure this is so.
Some cases were examined in an earlier
paper \cite{BooreFal2013} which showed that, for a class of strongly
connected digraphs, it is possible to construct CSSC GD-IFSs on $\mathbb{R}$
with attractors that cannot be obtained from a standard IFS, with or without
the CSSC. Another paper \cite{Boore2014} uses a different argument to
construct CSSC GD-IFSs on $\mathbb{R}$ with attractors that cannot be
obtained from a standard IFS. This paper further investigates this issue for all
\emph{strongly connected} digraphs (or even wider classes of digraphs).

For a strongly connected digraph $G$, it is known in \cite[Lemma 5.1]%
{Boore2014} (see also Theorem \ref{5.1} in the Appendix) that, if all directed
circuits in $G$ go through a vertex $u\in V$, then for any (COSC) GD-IFS
based on $G$, its attractor $F_{u}$ is also the attractor of a (COSC)
standard IFS. By way of contrast, we will show that if, for some vertex $u\in V$,
not all directed circuits in $G$ go through  $u$, then it is possible to define
GD-IFSs of similarities satisfying the COSC so that the corresponding
attractor $F_{u}$ is not the attractor of a standard IFS of similarities
satisfying the COSC (Lemma \ref{main}). Moreover, this is true for `almost
all' choices of similarities in a natural sense (Theorem \ref{COSC except}).
The proof basically relies on identifying a characteristic of the `gap
length set', where we use a shorter systematical algebraic argument `ratio
analysis' rather than the categorising method of \cite[Section 6]%
{BooreFal2013} which only works for certain classes of digraphs. In fact we
can relax the strong connectivity of $G$ in this construction (Lemma \ref%
{main thm}) and the `ratio analysis' method may have further applications to
other related problems. We finally apply \cite[Theorem 1.4]{Boore2014} (see
also Theorem \ref{App} in the Appendix) to show immediately that, there exists
GD-IFSs of similarities with the CSSC so that the corresponding attractor $%
F_{u}$ is not the attractor of a standard IFS.

GD-IFSs considered in this paper are \textit{inhomogeneous}, by which we
mean GD-IFSs of contracting similarities with not all contraction ratios
equal. We will require the COSC condition, which is easy to verify from the
parameters of a GD-IFS by solving simultaneous linear inequalities. There
are difficulties in relaxing this condition to OSC (even in $\mathbb{R}$)
where many problems still remain open even for standard IFSs, such as the
affine-embedding problem \cite[Conjecture 1.1]{FengXiong2018} or the inverse
fractal problem (determining the \emph{generating }IFSs of a standard IFS
attractor) \cite{FengWang2009}. The question considered here can be viewed
as an inverse-type problem, where we show certain GD-attractors have no
generating standard IFS (with or without the COSC). Previous results on
inhomogeneous self-similar sets also require this condition \cite[Section 4]%
{FengWang2009} or stronger conditions such as SSC and restrictions on
Hausdorff dimension \cite{Algom2018,Elekes2010,FengXiong2018}. Thus one
might expect similar difficulties for inhomogeneous GD-attractors.

This paper is organised as follows. In Section \ref{Sect2}, we first
introduce and obtain an expression for the gap length set of COSC
GD-attractors, and we then introduce our algebraic method `ratio analysis',
and derive a key lemma (Lemma \ref{Lemma 2.9}) relating the ratio sets of
GD-IFSs and standard IFSs with the COSC. In Section \ref{Sect3} we introduce
natural vector sets and construct GD-IFSs satisfying the COSC or the CSSC.
In Section \ref{Sect4} we use the GD-IFSs constructed in Section \ref{Sect3}
to show that the corresponding GD-attractors are not the attractors of COSC
standard IFSs using both the `ratio analysis' lemmas and the tool developed
in \cite{Boore2014}. We provide some examples to illustrate our assertions.

\section{Gap length sets and ratio analysis}

\label{Sect2}

\subsection{Gap length sets}

For a compact set $K\subset \mathbb{R}$ with $(\mathrm{conv}\,K)\setminus
K\neq \emptyset $, let
\begin{equation}
(\mathrm{conv}\,K)\setminus K=\mathop{\textstyle \bigcup }_{i}U_{i}
\label{21}
\end{equation}%
be the unique decomposition of the disjoint non-empty bounded complementary
intervals $\{U_{i}=(a_{i},b_{i})\}_{i}$ (see for example \cite[Chapter 2,
Theorem 9]{Pugh}), which will be called the \emph{gaps} of $K$ numbered by
decreasing length (and left to right for equal length intervals).

\begin{definition}[Gap length set]
Define the \emph{gap length set} of a compact set $K\subset \mathbb{R}$ to
be
\begin{equation*}
\mathrm{GL}(K):=\{b_{i}-a_{i}\}_{i}
\end{equation*}%
that is, the set of lengths of all the gaps of $K$. If $(\mathrm{conv}%
\,K)\setminus K=\emptyset $, that is, if $K$ is an interval (or a
singleton), we define $\mathrm{GL}(K):=\emptyset $.
\end{definition}

For each vertex $u\in V$, we arrange the edges leaving $u$, denoted by $%
e_{u}^{(k)}$ $(k=1,\mathbb{\cdots },d_{u})$ in the following way. Denote by $%
\omega (e)$ the terminal vertex of an edge $e\in E$, then the interiors of
the intervals $S_{e}(\mathrm{conv}\ F_{\omega (e)})$ are disjoint due to the
COSC. We rank these intervals in order from left to right, and denote the $k$%
th interval by
\begin{equation}
S_{u}^{(k)}\left( \mathrm{conv}\,F_{\omega (e_{u}^{(k)})}\right) \,\text{\ \
}(1\leq k\leq d_{u})  \label{sk2}
\end{equation}%
with the edges (and also the GD-IFS $\{S_{e}\}_{e\in E}$) arranged according
to this order.

\begin{definition}[Basic gaps]
\label{BG}With the above notation, for each $u\in V$ and $1\leq k\leq
d_{u}-1 $ ($d_{u}\geq 2$), let $\lambda _{u}^{(k)}$ be the length of the
complementary open interval between $S_{u}^{(k)}\left( \mathrm{conv}%
\,F_{\omega (e_{u}^{(k)})}\right) $ and $S_{u}^{(k+1)}\left( \mathrm{conv}%
\,F_{\omega (e_{u}^{(k+1)})}\right) \,$ (possibly $\lambda _{u}^{(k)}=0$).
All such complementary intervals (possibly empty) are called the \emph{basic
gaps} of this ordered COSC GD-IFS $\{S_{u}^{(k)}\}$ sitting at vertex $u$.
Let
\begin{equation}
\Lambda _{u}:=\{\lambda _{u}^{(k)}:\lambda _{u}^{(k)}>0,1\leq k\leq
d_{u}-1\},  \label{14}
\end{equation}%
be the set of strictly positive lengths of the basic gaps associated with
vertex $u\in V$, see Figure \ref{Basic}.
\end{definition}

\begin{figure}[th]
\begin{center}
\includegraphics[width=12cm, height=2.1cm]{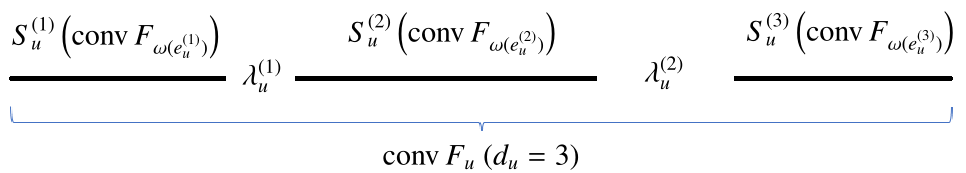}
\end{center}
\caption{Basic gaps of $F_{u}$.}
\label{Basic}
\end{figure}

As standard IFSs are one-vertex GD-IFSs, this definition is also applicable
to standard IFSs when we will omit the single vertex.

The GD-attractors $(F_{u})_{u\in V}$ of any GD-IFS can be determined in the
following way, see \cite[Equation (15)]{MauldinWil1988}. For any list of
compact sets $(I_{u})_{u\in V}$, we define
\begin{equation}
\text{\ }I_{u}^{m}:=\mathop{\textstyle \bigcup }_{\mathbf{e}\in E_{u}^{m}}S_{%
\mathbf{e}}\left( I_{\omega (\mathbf{e})}\right) \text{\ for any }m\geq 1,
\label{000}
\end{equation}%
where $E_{u}^{m}$ denotes the set of paths of length $m$ leaving $u$ and $%
\omega (\mathbf{e})$ denotes the terminal vertex of path $\mathbf{e}$. Note
that if
\begin{equation}
I_{u}^{1}\subset I_{u}\text{ for each }u\in V,  \label{001}
\end{equation}%
then the sequence $I_{u}^{m}$ decreases in $m$ in the sense that $%
I_{u}^{m+1}\subseteq I_{u}^{m}$ for every $m\geq 1$, since
\begin{eqnarray}
I_{u}^{m+1} &=&\mathop{\textstyle \bigcup }_{\widetilde{\mathbf{e}}\in
E_{u}^{m+1}}S_{\widetilde{\mathbf{e}}}\left( I_{\omega (\widetilde{\mathbf{e}%
})}\right) =\mathop{\textstyle \bigcup }_{\mathbf{e}\in E_{u}^{m}}%
\mathop{\textstyle \bigcup }_{e\in E_{\omega (\mathbf{e})}^{1}}S_{\mathbf{e}%
}\circ S_{e}\left( I_{\omega (e)}\right)  \notag \\
&=&\mathop{\textstyle \bigcup }_{\mathbf{e}\in E_{u}^{m}}S_{\mathbf{e}%
}\left( \mathop{\textstyle \bigcup }_{e\in E_{\omega (\mathbf{e}%
)}^{1}}S_{e}\left( I_{\omega (e)}\right) \right) =\mathop{\textstyle \bigcup
}_{\mathbf{e}\in E_{u}^{m}}S_{\mathbf{e}}\left( I_{\omega (\mathbf{e}%
)}^{1}\right)  \label{002-2} \\
&\subseteq &\mathop{\textstyle \bigcup }_{\mathbf{e}\in E_{u}^{m}}S_{\mathbf{%
e}}\left( I_{\omega (\mathbf{e})}\right) =I_{u}^{m}.  \label{002}
\end{eqnarray}%
From this, it is known that for each $u\in V$,
\begin{equation}
F_{u}=\tbigcap_{m=1}^{\infty }I_{u}^{m},  \label{003}
\end{equation}%
provided that (\ref{001}) is satisfied.

In particular, taking $I_{u}=\mathrm{conv}\,F_{u}$ for each $u\in V$, we see
that (\ref{001}) is satisfied, since by (\ref{gdattract})%
\begin{equation}
F_{u}\subseteq \mathop{\textstyle \bigcup }_{v\in V}\mathop{\textstyle
\bigcup }_{e\in E_{uv}}S_{e}(\mathrm{conv}\,F_{v})=I_{u}^{1}\subseteq
\mathrm{conv\Big(}\mathop{\textstyle \bigcup }_{v\in V}\mathop{\textstyle
\bigcup }_{e\in E_{uv}}S_{e}(F_{v})\Big)=\mathrm{conv}\,F_{u}=I_{u}.
\label{006}
\end{equation}%
In this case, the (\ref{003}) is true. Moreover, by taking convex hulls in (%
\ref{006}), we know that%
\begin{equation*}
\mathrm{conv}\,F_{u}\subseteq \mathrm{conv}\,I_{u}^{1}\subseteq \mathrm{conv}%
\,\mathrm{conv}\,F_{u}=\mathrm{conv}\,F_{u},\,
\end{equation*}%
which gives that
\begin{equation}
\mathrm{conv}\,I_{u}^{1}=\mathrm{conv}\,F_{u}=I_{u},  \label{convex}
\end{equation}%
meaning that the two endpoints of the interval $\mathrm{conv}\,I_{u}^{1}$
coincide with those of the interval $\mathrm{conv}\,F_{u}=I_{u}$. This fact
will be used shortly.

Throughout this paper, the product $AB$ of sets $A,\ B\subset \mathbb{R}$ is defined to be
$AB=\{ab:\ a\in A,\ b\in B\},$ and when we encounter the product of a
set in $\mathbb{R}$ and a constant, regard the constant as a set in $\mathbb{%
R}$. If $A$ is an empty set then $AB$ is also empty.

The following proposition gives a characterization for the gap length set of
an attractor $F_{u}$ of any COSC GD-IFS, which slightly extends a result in
\cite[below equation (5.2) in Section 5]{BooreFal2013} to the case when a
GD-IFS satisfies the COSC.

\begin{proposition}
\label{P2}Let $(V,E)$ be a digraph with $d_{u}\geq 2$ for all $u\in V$, and
let $F_{u}$ be a GD-attractor of a GD-IFS in $\mathbb{R}$ with the COSC
based on $(V,E)$. With the above notation, the gap length set $\mathrm{GL}%
(F_{u})$ of the attractor $F_{u}$ is given by
\begin{equation}
\mathrm{GL}(F_{u})=\Lambda _{u}\mathop{\textstyle \bigcup } \Big(%
\mathop{\textstyle \bigcup }\limits_{m=1}^{\infty }\mathop{\textstyle
\bigcup }_{v\in V}\Lambda _{v}\big\{|\rho _{\mathbf{e}}|:\mathbf{e}\text{ is
a directed path from $u$ to $v$ with length }m\big\}\Big).  \label{gls}
\end{equation}%
When there is no directed path from $u$ to $v$, the set $\{|\rho _{\mathbf{e}%
}|\}$ is understood to be empty.
\end{proposition}

\begin{proof}
When $\mathrm{GL}(F_{u})=\emptyset $, that is, $F_{u}=\mathrm{conv}\,F_{u}$,
we have for all $m\geq 1$
\begin{equation*}
\mathrm{conv}\,F_{u}\supseteq \mathop{\textstyle \bigcup }\limits_{v\in V}%
\mathop{\textstyle \bigcup }\limits_{\mathbf{e}\in E_{uv}^{m}}S_{\mathbf{e}}(%
\mathrm{conv}\,F_{v})\supseteq \mathop{\textstyle \bigcup }\limits_{v\in V}%
\mathop{\textstyle \bigcup }\limits_{\mathbf{e}\in E_{uv}^{m}}S_{\mathbf{e}%
}(F_{v})=F_{u}=\mathrm{conv}\,F_{u},
\end{equation*}%
where $E_{uv}^{m}$ is a collection of paths from vertex $u$ to vertex $v$
with length $m$. From this and using the COSC, we see that
\begin{equation*}
S_{\mathbf{e}}(\mathrm{conv}\,F_{v})=S_{\mathbf{e}}(F_{v})
\end{equation*}%
for every $v\in V$ and every directed path $\mathbf{e}$ of length $m$ from $%
u $ to $v$, showing that $\Lambda _{v}=\emptyset $ for all $v\in V$ to which
a directed path from $u$ exists. Thus (\ref{gls}) is trivial in this case.

In the sequel, we assume that $\mathrm{GL}(F_{u})\neq \emptyset $. Let $u\in
V$ be a vertex. Set $I_{u}:=\mathrm{conv}\,F_{u}$ for each $u\in V$, and (%
\ref{003}) holds true by virtue of (\ref{006}). So the gaps of $F_{u}$ will
be given by
\begin{equation}
(\mathrm{conv}\,F_{u})\setminus F_{u}=I_{u}\setminus \left(
\tbigcap_{m=1}^{\infty }I_{u}^{m}\right) =\left( I_{u}\setminus
I_{u}^{1}\right) \mathop{\textstyle \bigcup } \left( \mathop{\textstyle
\bigcup }_{m=1}^{\infty }I_{u}^{m}\setminus I_{u}^{m+1}\right) ,  \label{004}
\end{equation}%
which consists of the complementary open intervals in $I_{u}\setminus
I_{u}^{1}$ and $I_{u}^{m}\setminus I_{u}^{m+1}$ ($1\leq m<\infty $). We need
to calculate the lengths of these open intervals.

Indeed, for the open set $I_{u}\setminus I_{u}^{1}$, we know by definition (%
\ref{000}) that%
\begin{equation}
I_{u}\setminus I_{u}^{1}=\mathrm{conv}\,F_{u}\setminus \mathop{\textstyle
\bigcup }_{\mathbf{e}\in E_{u}^{1}}S_{\mathbf{e}}\left( I_{\omega (\mathbf{e}%
)}\right) =\mathop{\textstyle \bigcup }_{i=1}^{d_{u}-1}G_{u}^{(r)},
\label{005}
\end{equation}%
where $G_{u}^{(r)}$ for $1\leq r\leq d_{u}-1$ form the basic gaps of $F_{u}$%
, whose lengths form the set $\Lambda _{u}$ by using (\ref{convex}) with the
property that two intervals $I_{u}$ and $I_{u}^{1}$ have the same endpoints,
see Figure \ref{Basic}).

On the other hand, for any $m\geq 1$, due to the COSC, the interiors of the
level-$m$ intervals $\left\{ S_{\mathbf{e}}\left( I_{\omega (\mathbf{e}%
)}\right) \right\} _{\mathbf{e}\in E_{u}^{m}}$ are disjoint for any $m$. We
know by (\ref{002-2}) that
\begin{equation}
I_{u}^{m}\setminus I_{u}^{m+1}=\mathop{\textstyle \bigcup }_{\mathbf{e}\in
E_{u}^{m}}S_{\mathbf{e}}\left( I_{\omega (\mathbf{e})}\setminus I_{\omega (%
\mathbf{e})}^{1}\right) =\mathop{\textstyle \bigcup }_{\mathbf{e}\in
E_{u}^{m}}S_{\mathbf{e}}\left( \mathop{\textstyle \bigcup }_{r=1}^{d_{\omega
(\mathbf{e})}-1}G_{\omega (\mathbf{e})}^{(r)}\right) \text{ \ (using (\ref%
{005})).}  \label{f9}
\end{equation}%
The above union consists of disjoint complementary open intervals $S_{\mathbf{%
e}}(G_{\omega (\mathbf{e})}^{(r)})$, whose lengths are given by $|\rho _{%
\mathbf{e}}|\cdot \lambda _{\omega (\mathbf{e})}^{(r)}$, which form the gap
length sets at the $m$th-level for any $m\geq 1$. Summing up over $m$ will
give the double union in the right-hand side of (\ref{gls}), and so (\ref%
{gls}) follows from (\ref{004}) and the definition of $\mathrm{GL}(F_{u})$.
\end{proof}

\subsection{Ratio analysis}

We will use ``ratio analysis''\ to analyse sets $\Theta $ of positive real
numbers in $(0,\infty )$, in terms of strictly decreasing geometric
sequences $\{\theta ^{\prime }r^{k}\}_{k=0}^{\infty }$ that are contained in
$\Theta $.

\begin{definition}
Let $\Theta \subset (0,\infty )$. For $\theta \in \Theta $, let
\begin{equation}
R_{\Theta }(\theta )=\{r\in (0,1):\text{ there exists some }\theta ^{\prime
}\in \Theta \text{ such that }\theta \in \{\theta ^{\prime
}r^{k}\}_{k=0}^{\infty }\subset \Theta \},  \label{R}
\end{equation}%
the set of common ratios of strictly decreasing geometric sequences in $%
\Theta $ that contains $\theta $ (the set $R_{\Theta }(\theta )$\ may be
empty).
\end{definition}

This concept arises quite naturally as the characteristic set GL($F_u$)
contains many geometric sequences. The following
definition will be used in studying $R_{\mathrm{GL}(F_u) }(\theta )$ later on.

\begin{definition}
For a finite set $A=\{a_{i}\}_{i=1}^{n}\subset (0,\infty )$, define $A^{%
\mathbb{Z}_{+}^{\ast }}$ (resp. $A^{\mathbb{Q}_{+}^{\ast }}$, $A^{\mathbb{Q}%
^{\ast }}$) to be the union of all products $\tprod%
\limits_{i=1}^{n}a_{i}^{m_{i}}$ where $(m_{i})_{i=1}^{n}$ are non-zero
vectors whose entries are nonnegative integers (resp. nonnegative rationals,
rationals). Let $A^{\mathbb{Z}_{+}}=\{1\}\cup A^{\mathbb{Z}_{+}^{\ast }}$,
that is, the union of all products $\tprod\limits_{i=1}^{n}a_{i}^{m_{i}}$
where $(m_{i})_{i}$ are nonnegative integer vectors (including the zero
vector). Similarly, $A^{\mathbb{Q}}=\{1\}\cup A^{\mathbb{Q}^{\ast }}$ and $%
A^{\mathbb{Q}_{+}}=\{1\}\cup A^{\mathbb{Q}_{+}^{\ast }}$.
\end{definition}

We will analyse GL($F_u$) given by (\ref{gls}) with the following Lemma.

\begin{lemma}
\label{GDR}Let $A=\{a_{i}\}_{i=1}^{n}\subset (0,1)$ for $n\in \mathbb{Z}%
_{+}^{\ast }:=\{1,2,\cdots \}$, and $\lambda _{j}\ (j=1,\cdots ,m)$ be
positive real numbers (not necessarily distinct). Let $\Theta =%
\mathop{\textstyle \bigcup }\limits_{j=1}^{m}\lambda _{j}A_{j}$ where $%
A_{j}\subset A^{\mathbb{Z}_{+}}$ for $1\leq j\leq m$.

$(i)$ Then $R_{\Theta }(\theta )\subset A^{\mathbb{Q}_{+}^{\ast }}$ for all $%
\theta \in \Theta $.

$(ii)$ If ${\lambda _{p}}/{\lambda _{q}}\notin A^{\mathbb{Q}}$ for all
distinct $p,q\in \{1,\mathbb{\cdots },m\}$ when $m\geq 2$, then for every
strictly decreasing geometric sequence $\{\theta ^{\prime
}r^{k}\}_{k=0}^{\infty }\subset \Theta $, there exists a unique $l\in \{1,%
\mathbb{\cdots },m\}$ such that $\{\theta ^{\prime }r^{k}\}_{k=0}^{\infty
}\subset \lambda _{l}A_{l}$, and
\begin{equation}
\theta ^{\prime }r^{k}\notin \lambda _{j}A_{j}\text{ \ for all }j\neq l\text{
and all }k\geq 0.  \label{51}
\end{equation}
\end{lemma}

Condition $(ii)$ in Lemma \ref{GDR} means that the sets $\{\lambda
_{j}A_{j}\}_{j=1}^{m}$ are disjoint.

\begin{proof}
$(i)$ Let $\theta \in \Theta $. Assume that $R_{\Theta }(\theta )\neq
\emptyset $. Let $r\in R_{\Theta }(\theta )$. By (\ref{R}), there exists $%
\theta ^{\prime }\in \Theta $ such that $\theta \in \{\theta ^{\prime
}r^{k}\}_{k=0}^{\infty }\subset \Theta $, so by the pigeonhole principle we
can find some $\lambda _{l}$ such that $\{\theta ^{\prime
}r^{k}\}_{k=0}^{\infty }\tbigcap \lambda _{l}A_{l}$ is infinite. Write this
infinite subsequence as
\begin{equation}
\theta ^{\prime }r^{k_{t}}=\lambda _{l}\tprod\limits_{i=1}^{n}a_{i}^{m_{i,t}}
\label{pq2}
\end{equation}%
where $(m_{i,t})_{i=1}^{n}\in \mathbb{Z}_{+}^{n}$ and $\{k_{t}\}\subset
\mathbb{Z}_{+}:=\{0,1,2,\cdots \}$ with $k_{t}<k_{t+1}$, for $t\in \mathbb{Z}%
_{+}$. Applying Proposition \ref{part} in the Appendix with $B=%
\{(m_{i,t})_{i=1}^{n}\}_{t\in \mathbb{Z}_{+}}$, there exist two distinct
vectors $(m_{i,p})_{i=1}^{n}$ and $(m_{i,q})_{i=1}^{n}$ in $\mathbb{Z}%
_{+}^{n}$ for some two indices $p<q$ in $\mathbb{Z}_{+}$, such that
\begin{equation}
(m_{i,p})_{i=1}^{n}\leq (m_{i,q})_{i=1}^{n}  \label{pq1}
\end{equation}%
under the partial order defined by inequality of all coordinates. Therefore,
we have by (\ref{pq2})
\begin{equation*}
r^{k_{q}-k_{p}}=\frac{\theta ^{\prime }r^{k_{q}}}{\theta ^{\prime }r^{k_{p}}}%
=\frac{\lambda _{l}\tprod\limits_{i=1}^{n}a_{i}^{m_{i,q}}}{\lambda
_{l}\tprod\limits_{i=1}^{n}a_{i}^{m_{i,p}}}=\tprod%
\limits_{i=1}^{n}a_{i}^{m_{i,q}-m_{i,p}}\text{ \ \big(or }r=\tprod%
\limits_{i=1}^{n}a_{i}^{(m_{i,q}-m_{i,p})/(k_{q}-k_{p})}\text{\big).}
\end{equation*}%
Since%
\begin{equation*}
\left( \frac{m_{i,q}-m_{i,p}}{k_{q}-k_{p}}\right) _{i=1}^{n}\in (\mathbb{Q}%
_{+}^{n})^{\ast }
\end{equation*}%
using (\ref{pq1}), it follows that $r\in A^{\mathbb{Q}_{+}^{\ast }}$ by
definition. Therefore,%
\begin{equation*}
R_{\Theta }(\theta )\subset A^{\mathbb{Q}_{+}^{\ast }}
\end{equation*}%
for all $\theta \in \Theta $, thus proving our assertion $(i)$.

$(ii)$ For $m\geq 2$, suppose that there exist distinct $p,q\in \{1,\mathbb{%
\cdots },m\}$ such that $\theta ^{\prime }r^{k}\in \lambda _{p}A_{p}$ and $%
\theta ^{\prime }r^{j}\in \lambda _{q}A_{q}$ for some $k,j\in \mathbb{Z}_{+}$%
. Write%
\begin{equation}
\theta ^{\prime }r^{k}=\lambda _{p}\tprod\limits_{i=1}^{n}a_{i}^{p_{i,k}}%
\text{ \ and \ }\theta ^{\prime }r^{j}=\lambda
_{q}\tprod\limits_{i=1}^{n}a_{i}^{q_{i,j}}\text{ \ (}p_{i,k},q_{i,j}\in
\mathbb{Z}_{+}\text{).}  \label{53}
\end{equation}%
By $(i)$, $r\in R_{\Theta }(\theta ^{\prime })\subset A^{\mathbb{Q}%
_{+}^{\ast }}$ since $\theta ^{\prime }\in \Theta $, and so $r^{k-j}\in A^{%
\mathbb{Q}}$. It follows that%
\begin{equation*}
\frac{\lambda _{p}}{\lambda _{q}}=r^{k-j}\tprod%
\limits_{i=1}^{n}a_{i}^{q_{i,j}-p_{i,k}}\in A^{\mathbb{Q}}A^{\mathbb{Q}}=A^{%
\mathbb{Q}},
\end{equation*}%
leading to a contradiction to our assumption. Thus, there exists a unique
integer $l\in \{1,\mathbb{\cdots },m\}$ such that $\{\theta ^{\prime
}r^{k}\}_{k=0}^{\infty }\subset \lambda _{l}A_{l}.$

It remains to show (\ref{51}). In fact, if (\ref{51}) were not true, then $%
\theta ^{\prime }r^{k}\in \lambda _{t}A_{t}$ for some integer $k\geq 0$ and
some $t\neq l$. Taking $p=l$, $j=k$, $q=t$ in (\ref{53}), we would have%
\begin{equation*}
\frac{\lambda _{l}}{\lambda _{t}}=\tprod%
\limits_{i=1}^{n}a_{i}^{q_{i,k}-p_{i,k}}\in A^{\mathbb{Q}},
\end{equation*}%
leading to a contradiction. The assertion (\ref{51}) follows.
\end{proof}

The following corollary will be used to describe a certain `homogeneity' property of
(the gap length sets of) attractors of COSC standard IFSs.

\begin{corollary}
\label{SSR} Let $X\subset (0,1)$ and $\Lambda \subset (0,\infty )$ be two
finite sets. Then
\begin{equation*}
X^{\mathbb{Z}_{+}^{\ast }}\subset R_{\Lambda X^{\mathbb{Z}_{+}}}(\theta
)\subset X^{\mathbb{Q}_{+}^{\ast }}
\end{equation*}%
for every $\theta \in \Lambda X^{\mathbb{Z}_{+}}$.
\end{corollary}

\begin{proof}
Let $\theta \in \Lambda X^{\mathbb{Z}_{+}}$. Since $X^{\mathbb{Z}_{+}^{\ast
}}\subset $ $X^{\mathbb{Z}_{+}}$, %
\begin{equation*}
\theta X^{\mathbb{Z}_{+}^{\ast }}\subset (\Lambda X^{\mathbb{Z}_{+}})X^{%
\mathbb{Z}_{+}}=\Lambda (X^{\mathbb{Z}_{+}}X^{\mathbb{Z}_{+}})=\Lambda X^{%
\mathbb{Z}_{+}}.
\end{equation*}%
For any $r\in X^{\mathbb{Z}_{+}^{\ast }}$ and $k\in \mathbb{Z}_{+}$, we have
$r^{k}\in X^{\mathbb{Z}_{+}}$ and so
\begin{equation*}
\theta r^{k}\in \left( \Lambda X^{\mathbb{Z}_{+}}\right) X^{\mathbb{Z}%
_{+}}=\Lambda X^{\mathbb{Z}_{+}},
\end{equation*}%
thus showing that $r\in R_{\Lambda X^{\mathbb{Z}_{+}}}(\theta )$ by
definition (\ref{R}) with $\Theta =\Lambda X^{\mathbb{Z}_{+}}$, so the first
inclusion follows.

The second inclusion also follows by taking $A=X$, $\lambda _{j}\in \Lambda $
and each $A_{j}=X^{\mathbb{Z}_{+}}$ in Lemma \ref{GDR}$(i)$ (so that $\Theta
=\Lambda X^{\mathbb{Z}_{+}}$).
\end{proof}

As an application of Lemma \ref{GDR} and Corollary \ref{SSR}, we derive a
key lemma that will be used to distinguish the attractor of a COSC GD-IFS
from that of a COSC standard IFS.

\begin{definition}[Absolute contraction ratio set]
The \emph{absolute contraction ratio set} of a GD-IFS is defined to be the
set of the absolute values of the contraction ratios of the similarities,
that is $\{|\rho _{e}|:e\in E\}$.
\end{definition}

\begin{lemma}
\label{Lemma 2.9}Let $(F_{v})_{v\in V}$ be the attractors of a COSC GD-IFS
based on a digraph with $d_{v}\geq 2$ for all $v\in V$, with absolute
contraction ratio set $A$. Assume that for some $u$, the set $F_{u}$ is not
an interval (or a singleton) and is the attractor of some COSC standard IFS
with absolute contraction ratio set $X$.

\noindent $(i)$ Then for all $\theta \in \mathrm{GL}(F_{u})$
\begin{equation}
X\subset X^{\mathbb{Z}_{+}^{\ast }}\subset \ R_{\mathrm{GL}(F_{u})}(\theta
)\ \subset \ A^{\mathbb{Q}_{+}^{\ast }}\cap X^{\mathbb{Q}_{+}^{\ast }}.
\label{incl}
\end{equation}

\noindent $(ii)$ If $A_{1}\cup A_{2}=A$ and $A_{1}^{\mathbb{Q}^{\ast }}\cap
A_{2}^{{}^{\mathbb{Q}_{+}^{\ast }}}=\emptyset $, then the following
dichotomy is true: \emph{either}
\begin{equation}
R_{\mathrm{GL}(F_{u})}(\theta )\cap A_{1}^{\mathbb{Q}_{+}^{\ast }}\neq
\emptyset \text{ for all }\theta \in \mathrm{GL}(F_{u}),  \label{31}
\end{equation}%
\emph{or}
\begin{equation}
R_{\mathrm{GL}(F_{u})}(\theta )\cap A_{1}^{\mathbb{Q}_{+}^{\ast }}=\emptyset
\text{ \ for all }\theta \in \mathrm{GL}(F_{u}).  \label{32}
\end{equation}
\end{lemma}

The assertion $(ii)$ of Lemma \ref{Lemma 2.9} gives a necessary condition
that a COSC GD-attractor $F_{u}$ is also the attractor of some COSC standard
IFS in the following way: if there exists two elements $\theta _{1},\theta
_{2}\in \mathrm{GL}(F_{u})$ such that (\ref{31}) holds for $\theta _{1}$
whilst (\ref{32}) holds for $\theta _{2}$, then $F_{u}$ is not the attractor
of any COSC standard IFS. This assertion will be used in Lemma \ref{main thm}
below.

\begin{proof}
$(i)$ Let $\Lambda $ be the set of nonzero basic gap lengths of some COSC
standard IFS with the attractor $F_{u}$, and let $X$ be the absolute
contraction ratio set. Regard this standard IFS as a GD-IFS based on $%
(\{v\},\{e_{j}\}_{j=1}^{m})$ where $e_{j}$ are loops of the single vertex $v$%
, all directed paths of length $k\geq 1$ are now $e_{i_{1}}e_{i_{2}}\cdots
e_{i_{k}}$ where $i_{l}=1,2,\cdots ,m$ for all $l=1,2,\cdots ,k$. By (\ref%
{gls}),
\begin{eqnarray*}
\mathrm{GL}(F_{u}) &=&\Lambda \mathop{\textstyle \bigcup } \Big( %
\mathop{\textstyle \bigcup }\limits_{m=1}^{\infty }\Lambda \big\{|\rho _{%
\mathbf{e}}|:\mathbf{e}\text{ is a directed path from }v\text{ to $v$ with
length }m\big\}\Big) \\
&=&\Lambda \cup \Lambda X^{\mathbb{Z}_{+}^{\ast }}=\Lambda X^{\mathbb{Z}%
_{+}}.
\end{eqnarray*}%
Note that $\mathrm{GL}(F_{u})$ is non-empty by using our assumption that $%
F_{u}$ is not an interval or a singleton.

On the other hand, Corollary \ref{SSR} implies that
\begin{equation}
X^{\mathbb{Z}_{+}^{\ast }}\subset R_{\mathrm{GL}(F_{u})}(\theta )=R_{\Lambda
X^{\mathbb{Z}_{+}}}(\theta )\subset X^{\mathbb{Q}_{+}^{\ast }}  \label{71}
\end{equation}%
for all $\theta \in \mathrm{GL}(F_{u})$.
Recall that a directed circuit containing $u$ is a directed path from $u$ to
$u$. We write the union given by (\ref{gls}) as
\begin{eqnarray}
\Theta \coloneqq\mathrm{GL}(F_{u}) &=&\Big(\mathop{\textstyle \bigcup }%
_{\lambda \in \Lambda _{u}}\lambda \left( \{1\}\mathop{\textstyle \bigcup }%
\big\{|\rho _{\mathbf{e}}|:\mathbf{e}\text{ is a directed circuit containing
$u$}\big\}\right) \Big)  \notag \\
&\;&\mathop{\textstyle \bigcup }\Big(\mathop{\textstyle \bigcup }_{v\in
V\setminus \{u\}}\mathop{\textstyle \bigcup }_{\lambda \in \Lambda
_{v}}\lambda \big\{|\rho _{\mathbf{e}}|:\mathbf{e}\text{ is a directed path
from $u$ to $v$}\big\}\Big).  \label{doubleunion}
\end{eqnarray}%
Since the absolute contraction ratios are all in $A$ (so that $|\rho _{%
\mathbf{e}}|\in A^{\mathbb{Z}_{+}}$), it follows from Lemma \ref{GDR}($i$)
that $R_{\mathrm{GL}(F_{u})}(\theta )\subset A^{\mathbb{Q}_{+}^{\ast }}$ for
all $\theta \in \mathrm{GL}(F_{u})$, which combines with (\ref{71}) to give
that%
\begin{equation*}
R_{\mathrm{GL}(F_{u})}(\theta )\in A^{\mathbb{Q}_{+}^{\ast }}\cap X^{\mathbb{%
Q}_{+}^{\ast }},
\end{equation*}%
leading to the inclusions in (\ref{incl}), as desired.

$(ii)$ If $X\cap A_{1}^{\mathbb{Q}_{+}^{\ast }}\neq \emptyset $, it follows
from (\ref{incl}) that
\begin{equation*}
X\cap A_{1}^{\mathbb{Q}_{+}^{\ast }}\newline
\subset R_{\mathrm{GL}(F_{u})}(\theta )\cap A_{1}^{\mathbb{Q}_{+}^{\ast }}
\end{equation*}%
for all $\theta \in \mathrm{GL}(F_{u})$, thus showing that (\ref{31}) is
true.

Now assume that $X\cap A_{1}^{\mathbb{Q}_{+}^{\ast }}=\emptyset $%
. We will show that (\ref{32}) is true.

We first claim that $A^{\mathbb{Q}_{+}^{\ast }}$ is the union of two
disjoint sets $A_{1}^{\mathbb{Q}_{+}^{\ast }}$ and $A_{1}^{\mathbb{Q}%
_{+}}A_{2}^{\mathbb{Q}_{+}^{\ast }}$. To see this, as $A_{1}^{\mathbb{Q}%
^{\ast }}\cap A_{2}^{{}^{\mathbb{Q}_{+}^{\ast }}}=\emptyset $ by assumption,
it follows that
\begin{equation}
A_{1}^{\mathbb{Q}_{+}^{\ast }}\cap A_{1}^{\mathbb{Q}_{+}}A_{2}^{\mathbb{Q}%
_{+}^{\ast }}=\emptyset .  \label{2.15}
\end{equation}%
In fact, if (\ref{2.15}) were not true, there would exist three elements$\ $%
\begin{equation*}
a\in A_{1}^{\mathbb{Q}_{+}^{\ast }},\ b\in A_{1}^{\mathbb{Q}_{+}},\ c\in
A_{2}^{\mathbb{Q}_{+}^{\ast }}
\end{equation*}%
with $a=bc$, from which $\frac{a}{b}\in A_{1}^{\mathbb{Q}}$ and $%
\frac{a}{b}=c\in A_{2}^{\mathbb{Q}_{+}^{\ast }}$. As $A_{1}^{\mathbb{Q}%
}=\{1\}\cup A_{1}^{\mathbb{Q}^{\ast }}$ by definition and $\{1\}\cap A_{2}^{%
\mathbb{Q}_{+}^{\ast }}=\emptyset $ due to $A_{2}^{\mathbb{Q}_{+}^{\ast
}}\subset (0,1)$, we see that
\begin{equation*}
\frac{a}{b}\in A_{1}^{\mathbb{Q}}\cap A_{2}^{\mathbb{Q}_{+}^{\ast }}=\left(
\{1\}\cup A_{1}^{\mathbb{Q}^{\ast }}\right) \cap A_{2}^{\mathbb{Q}_{+}^{\ast
}}=\left( \{1\}\cap A_{2}^{\mathbb{Q}_{+}^{\ast }}\right) \cup \left( A_{1}^{%
\mathbb{Q}^{\ast }}\cap A_{2}^{{}^{\mathbb{Q}_{+}^{\ast }}}\right)
=\emptyset ,
\end{equation*}%
a contradiction.

We need to show
\begin{equation}
A^{\mathbb{Q}_{+}^{\ast }}=A_{1}^{\mathbb{Q}_{+}^{\ast }}\cup A_{1}^{\mathbb{%
Q}_{+}}A_{2}^{\mathbb{Q}_{+}^{\ast }}.  \label{33}
\end{equation}%
In fact, let
\begin{equation*}
A_{1}=\{b_{i}\}_{i=1}^{m},\ A_{2}=\{c_{j}\}_{j=1}^{n}.
\end{equation*}%
As $A_{1}\cap A_{2}\subset A_{1}^{\mathbb{Q}^{\ast }}\cap A_{2}^{\mathbb{Q}%
_{+}^{\ast }}=\emptyset $, any element $a\in A^{\mathbb{Q}_{+}^{\ast
}}=(A_{1}\cup A_{2})^{\mathbb{Q}_{+}^{\ast }}$ can be written as
\begin{equation*}
a=\tprod\limits_{i=1}^{m}b_{i}^{p_{i}}\tprod\limits_{j=1}^{n}c_{j}^{q_{j}}%
\text{ }\ \text{for some }(p_{i})_{i=1}^{m}\in \mathbb{Q}_{+}^{m}\text{ and }%
(q_{j})_{j=1}^{n}\in \mathbb{Q}_{+}^{n},
\end{equation*}%
where not all $p_{i},q_{j}$ are zero. Thus, if all $q_{j}$ are zero,
then $a=\tprod\limits_{i=1}^{m}b_{i}^{p_{i}}\in A_{1}^{\mathbb{Q}_{+}^{\ast
}}$; otherwise $a\in A_{1}^{\mathbb{Q}_{+}}A_{2}^{\mathbb{Q}_{+}^{\ast }}$.
This proves (\ref{33}) by using (\ref{2.15}).

As $X\cap A_{1}^{\mathbb{Q}_{+}^{\ast }}=\emptyset $ and%
\begin{equation*}
X\subset A^{\mathbb{Q}_{+}^{\ast }}=A_{1}^{\mathbb{Q}_{+}^{\ast }}\cup
A_{1}^{\mathbb{Q}_{+}}A_{2}^{\mathbb{Q}_{+}^{\ast }}
\end{equation*}%
by using (\ref{incl}) and (\ref{33}), we have
\begin{equation}
X\subset A_{1}^{\mathbb{Q}_{+}}A_{2}^{\mathbb{Q}_{+}^{\ast }}.  \label{35}
\end{equation}%
We will show the following inclusion
\begin{equation}
X^{\mathbb{Q}_{+}^{\ast }}\subset A_{1}^{\mathbb{Q}_{+}}A_{2}^{\mathbb{Q}%
_{+}^{\ast }}.  \label{34}
\end{equation}%
Since $X$ is finite, let
\begin{equation*}
X=\{x_{l}\}_{l=1}^{k},\ A_{1}=\{b_{i}\}_{i=1}^{m},\
A_{2}=\{c_{j}\}_{j=1}^{n}.
\end{equation*}%
By (\ref{35}), we write for each $l=1,2,\cdots ,k$
\begin{equation*}
x_{l}=\tprod\limits_{i=1}^{m}b_{i}^{p_{i,l}}\tprod%
\limits_{j=1}^{n}c_{j}^{q_{j,l}}\text{ \ for some }(p_{i,l})_{i=1}^{m}\in
\mathbb{Q}_{+}^{m}\text{ and }(q_{j,l})_{j=1}^{n}\in (\mathbb{Q}%
_{+}^{n})^{\ast }.
\end{equation*}%
Then any element $x\in X^{\mathbb{Q}_{+}^{\ast }}$ can be written as
\begin{equation*}
x=\tprod_{l=1}^{k}x_{l}^{r_{l}}=\tprod_{l=1}^{k}\left(
\tprod\limits_{i=1}^{m}b_{i}^{p_{i,l}}\tprod\limits_{j=1}^{n}c_{j}^{q_{j,l}}%
\right)
^{r_{l}}=\tprod\limits_{i=1}^{m}b_{i}^{\sum_{l=1}^{k}p_{i,l}r_{l}}\tprod%
\limits_{j=1}^{n}c_{j}^{\sum_{l=1}^{k}q_{j,l}r_{l}}\text{ \ for some }%
(r_{l})_{l=1}^{k}\in \left( \mathbb{Q}_{+}^{k}\right) ^{\ast }.
\end{equation*}%
Note that the numbers $\sum_{l=1}^{k}p_{i,l}r_{l}$ and $%
\sum_{l=1}^{k}q_{j,l}r_{l}$ all belong to $\mathbb{Q}_{+}$. Since $%
r_{l^{\prime }}>0$ for some $l^{\prime }$ while $q_{j^{\prime },l^{\prime
}}>0$ for this $l^{\prime }$ and some $j^{\prime }$, we have $%
\sum_{l=1}^{k}q_{j^{\prime },l}r_{l}>0$ for this $j^{\prime }$. Therefore,
we obtain (\ref{34}).

Finally, by (\ref{incl}) and (\ref{34}), we have for all $\theta \in \mathrm{%
GL}(F_{u})$,
\begin{equation*}
R_{\mathrm{GL}(F_{u})}(\theta )\subset X^{\mathbb{Q}_{+}^{\ast }}\subset
A_{1}^{\mathbb{Q}_{+}}A_{2}^{\mathbb{Q}_{+}^{\ast }},
\end{equation*}%
from which, we easily conclude that (\ref{32}) holds by using (\ref{2.15}).
\end{proof}

\section{Construction of GD-IFSs}

\label{Sect3}We will construct COSC (CSSC) GD-IFSs in terms of vector sets
in Euclidean spaces, to analyse the existence and extent of \emph{non-trivial%
} GD-IFSs whose attractors are not attractors of any
(COSC) standard IFS.

For a digraph $G=(V,E)$ with $d_{i}\geq 2$ for $i\in V=\{1,2,\cdots ,N\}$,
we set
\begin{equation}
n:=2\#E-\#V=2(d_{1}+d_{2}+\cdots +d_{N})-N  \label{nn}
\end{equation}%
so that $n\geq N$ (recall that $d_{i}$ denotes the number of the edges
leaving vertex $i$). Define the subset $P_{0}$ in the Euclidean space $%
%TCIMACRO{\U{211d} }%
%BeginExpansion
\mathbb{R}
%EndExpansion
^{n}$, with $n$ given in (\ref{nn}), by%
\begin{eqnarray}
&&P_{0}\coloneqq\left\{ x=(x_{1}^{(1)},\mathbb{\cdots }%
,x_{1}^{(d_{1})},x_{2}^{(1)},\mathbb{\cdots },x_{2}^{(d_{2})},\cdots
,x_{N}^{(1)},\mathbb{\cdots },x_{N}^{(d_{N})},\xi _{1}^{(1)},\mathbb{\cdots }%
,\xi _{1}^{(d_{1}-1)},\cdots ,\xi _{N}^{(1)},\mathbb{\cdots },\xi
_{N}^{(d_{N}-1)})\right.  \notag \\
&&\left. \text{ where }x_{i}^{(k)},x_{i}^{(d_{i})}\in (-1,1)\setminus \{0\}%
\text{ and }\xi _{i}^{(k)}\geq 0\text{ for each vertex }i\in V\text{, }1\leq
k\leq d_{i}-1\right\} .  \label{pp}
\end{eqnarray}%
Each vector $x$ in $P_{0}$ consists of two kinds of entries: the entries $%
\{x_{i}^{(k)}\}_{i\in V,1\leq k\leq d_{i}}$ all lie in the set $%
(-1,1)\setminus \{0\}$, and will specify the contraction ratios of GD-IFSs to
be constructed, whilst the other entries $\{\xi _{i}^{(k)}\}_{i\in V,1\leq
k\leq d_{i}-1}$ are all non-negative, and will specify the basic gap lengths.

For vertex $i\in V$, let $\{e_{i}(k):1\leq k\leq d_{i}\}$ be the set of
edges leaving $i$, which are arranged in some order which will henceforth
remain fixed. For a point $x$ in $P_{0}$, we look at its entries $%
\{x_{i}^{(k)}\}_{i\in V,1\leq k\leq d_{i}}$ and define an $N\times N$ matrix
$M_{x}(s)$ for any $s>0$ by
\begin{equation}
M_{x}(s)=\left( M_{ij}(s)\right) _{1\leq i,j\leq N},  \label{mx}
\end{equation}%
where%
\begin{equation}
M_{ij}(s)=\sum\limits_{e_{i}(k)\in E_{ij}}|x_{i}^{(k)}|^{s}  \label{mij}
\end{equation}%
if $E_{ij}\neq \emptyset $, and $M_{ij}(s)=0$ if $E_{ij}=\emptyset $ (recall
that $E_{ij}$ is the set of (multiple) edges from vertex $i$ to vertex
$j$).

Let $\mathcal{b}$, $\ell $ be two vectors defined by
\begin{eqnarray}
\mathcal{b} &\coloneqq&(b_{i}^{(1)})_{i\in V}\text{ \ where }b_{i}^{(1)}\in
\mathbb{R},  \label{bb2} \\
\ell &\coloneqq&(l_{i})_{i\in V}\text{ \ where }l_{i}\geq 0.  \label{LL2}
\end{eqnarray}%
For each edge $e_{i}(k)$ ($1\leq k\leq d_{i}$) leaving vertex $i\in V$, we
define the mappings associated with a point $x$ in $P_{0}$ by
\begin{equation}
S_{e_{i}(k)}(t)=x_{i}^{(k)}(t-b_{\omega
(e_{i}(k))}^{(1)})+b_{i}^{(k)}-x_{i}^{(k)}l_{\omega (e_{i}(k))}\mathbf{1}%
_{\{x_{i}^{(k)}<0\}}\text{ \ for a variable }t\in
%TCIMACRO{\U{211d} }%
%BeginExpansion
\mathbb{R}
%EndExpansion
\text{,}  \label{S1}
\end{equation}%
where $\mathbf{1}_{\{x_{i}^{(k)}<0\}}=1$ if $x_{i}^{(k)}<0$, and $\mathbf{1}%
_{\{x_{i}^{(k)}<0\}}=0$ otherwise, and
\begin{equation}
b_{i}^{(k+1)}:=b_{i}^{(k)}+|x_{i}^{(k)}|l_{\omega (e_{i}(k))}+\xi _{i}^{(k)}%
\text{ \ for }i\in V\text{ and }1\leq k\leq d_{i}-1,  \label{555}
\end{equation}%
and $\omega (e_{i}(k))$ denotes the terminal vertex of the edge $%
e_{i}(k) $ as before.

Note that for any point $x\in P_{0}$, the mapping $S_{e_{i}(k)}$ defined as
in (\ref{S1}) has the contraction ratio $x_{i}^{(k)}\in (-1,1)\setminus
\{0\} $, therefore it is a contracting similarity, and
\begin{equation}
\digamma (x,\mathcal{b},\ell ):=\{S_{e_{i}(k)}:i\in V,1\leq k\leq d_{i}\}
\label{ff}
\end{equation}%
forms a GD-IFS on the digraph $(V,\{e_{i}(k)\})$, thus having a unique list
of GD-attractors $\{F_{i}\}_{i\in V}$.

For any two vectors $\mathcal{b}$, $\ell $ as in (\ref{bb2}), (\ref{LL2})
and any point $x$ in $P_{0}$, we define the closed intervals (which may be
singletons) for each vertex $i\in V$ by
\begin{eqnarray}
I_{i} &=&[b_{i}^{(1)},b_{i}^{(1)}+l_{i}]\text{,}  \label{I2} \\
I_{i}(k) &=&\big[b_{i}^{(k)},b_{i}^{(k)}+|x_{i}^{(k)}|l_{\omega (e_{i}(k))}\big]%
\text{ \ for }1\leq k\leq d_{i},  \label{I3}
\end{eqnarray}%
where $b_{i}^{(k+1)}$ for $1\leq k\leq d_{i}-1$ are given by (\ref{555}).

We will work with a subset $P$ of $P_{0}$ defined by%
\begin{equation}
P\coloneqq\Big\{x\in P_{0}:r_{\sigma }(M_{x}(1))<1,\sum_{k=1}^{d_{i}-1}\xi
_{i}^{(k)}>0\text{ \ for all \ }1\leq i\leq N\Big\},  \label{sm}
\end{equation}%
where the matrix $M_{x}(1)$ is defined by (\ref{mx}) with $s=1$, and $%
r_{\sigma }(M)$ denotes the \emph{spectral radius} of a matrix $M$, which is
the largest absolute value (complex modulus) of the eigenvalues of $M$.

We show that any point in $P$ will give arise to at least one COSC GD-IFS on
$G$, in form of (\ref{S1}), whose contraction ratios are $%
\{x_{i}^{(k)}\}_{i\in V,1\leq k\leq d_{i}}$ and whose attractor $F_{i}$ at
each vertex $i$ has the convex hull $I_{i}$ given by (\ref{I2}), having the
basic gap lengths $\{\xi _{i}^{(k)}\}_{1\leq k\leq d_{i}-1}$, provided that $%
l_{i}$ satisfies (\ref{lu}) below.

\begin{lemma}[Construction of GD-IFSs]
\label{P1}Let $G=(V,E)$ be a digraph with $d_{i}\geq 2$ for $i\in V$. With
the same notation above, let $x$ be any point in $P$ as in $(\ref{sm}%
)$ and $\mathcal{b}$ be any vector as in $(\ref{bb2})$. Let $(l_{i})_{i\in V}$
be a vector of real numbers given by
\begin{equation}
(l_{i})_{i\in V}^{T}:=(\mathrm{id}-M_{x}(1))^{-1}\left(
\sum_{k=1}^{d_{i}-1}\xi _{i}^{(k)}\right) _{i\in V}^{T},  \label{lu}
\end{equation}%
where $M^{T}$ denotes the transpose of a matrix $M$. Then any GD-IFS $%
\digamma (x,\mathcal{b})$, given by $(\ref{S1}), (\ref{ff})$ and $(\ref{lu})$ and
having attractors $\{F_{i}\}_{i\in V}$, satisfies the following properties.

\begin{enumerate}
\item[$(i).$] For each vertex $i\in V$, we have $l_{i}>0$ and
\begin{equation}
\mathrm{conv}\,F_{i}=I_{i}=[b_{i}^{(1)},\ b_{i}^{(1)}+l_{i}].  \label{C1}
\end{equation}

\item[$(ii).$] The GD-IFS $\digamma (x,\mathcal{b})$ satisfies the COSC. The
basic gaps of attractor $F_{i}$ for $i\in V$ are given by the following open
intervals in $%
%TCIMACRO{\U{211d} }%
%BeginExpansion
\mathbb{R}
%EndExpansion
$%
\begin{equation}
\left\{ \Big(b_{i}^{(k)}+|x_{i}^{(k)}|l_{\omega
(e_{i}(k))},b_{i}^{(k+1)}\Big)\right\} _{1\leq k\leq d_{i}-1},  \label{bg2}
\end{equation}%
which are arranged in order from left to right. The corresponding basic gap
lengths are
\begin{equation}
\left\{ b_{i}^{(k+1)}-\left( b_{i}^{(k)}+|x_{i}^{(k)}|l_{\omega
(e_{i}(k))}\right) =\xi _{i}^{(k)}\right\} _{1\leq k\leq d_{i}-1}.
\label{gl2}
\end{equation}%
If further all $\xi _{i}^{(k)}>0$ for $i\in V$ and $1\leq k\leq d_{i}-1$,
then $\digamma (x,\mathcal{b})$ satisfies the CSSC.
\end{enumerate}
\end{lemma}

\begin{proof}
Note that
\begin{equation}
b_{i}^{(d_{i})}=b_{i}^{(1)}+\sum_{k=1}^{d_{i}-1}(\xi
_{i}^{(k)}+|x_{i}^{(k)}|l_{\omega (e_{i}(k))}),  \label{99}
\end{equation}%
since, by repeatedly using definition (\ref{555}) of $b_{i}^{(k+1)}$,%
\begin{eqnarray*}
b_{i}^{(d_{i})} &=&b_{i}^{(d_{i}-1)}+|x_{i}^{(d_{i}-1)}|l_{\omega
(e_{i}(d_{i}-1))}+\xi _{i}^{(d_{i}-1)} \\
&=&\left( b_{i}^{(d_{i}-2)}+|x_{i}^{(d_{i}-2)}|l_{\omega
(e_{i}(d_{i}-2))}+\xi _{i}^{(d_{i}-2)}\right) +|x_{i}^{(d_{i}-1)}|l_{\omega
(e_{i}(d_{i}-1))}+\xi _{i}^{(d_{i}-1)}\\
&=&\cdots \\
&=&b_{i}^{(1)}+\sum_{k=1}^{d_{i}-1}(\xi _{i}^{(k)}+|x_{i}^{(k)}|l_{\omega
(e_{i}(k))}).
\end{eqnarray*}%
Also note that
\begin{equation}
l_{i}>0\text{ \ for each }i\in V,  \label{L1}
\end{equation}%
since, by using definition (\ref{sm}) of $P$, the matrix $(\mathrm{id}%
-M_{x}(1))$ is invertible and can be written as
\begin{equation*}
(\mathrm{id}-M_{x}(1))^{-1}=\mathrm{id}+M_{x}(1)+M_{x}^{2}(1)+\cdots ,
\end{equation*}%
 (see for example \cite[Lemma \textbf{B}.1, Appendix B]{Seneta}), %
from which it follows by definition (\ref{lu}) that
\begin{eqnarray}
l_{i} &=&\sum_{j\in V}\left( (\mathrm{id}-M_{x}(1))^{-1}\right) _{ij}\left(
\sum_{k=1}^{d_{j}-1}\xi _{j}^{(k)}\right)  \notag \\
&=&\sum_{j\in V}\left( \mathrm{id}+M_{x}(1)+M_{x}^{2}(1)+\mathbb{\cdots }%
\right) _{ij}\left( \sum_{k=1}^{d_{j}-1}\xi _{j}^{(k)}\right)  \notag \\
&\geq &\sum_{k=1}^{d_{i}-1}\xi _{i}^{(k)}>0  \label{LL6}
\end{eqnarray}%
by using the fact that $M_{x}(1)$ is a nonnegative matrix and that $%
\sum_{k=1}^{d_{i}-1}\xi _{i}^{(k)}>0$ by (\ref{sm}).

We claim that
\begin{equation}
b_{i}^{(d_{i})}+|x_{i}^{(d_{i})}|l_{\omega (e_{i}(d_{i}))}=b_{i}^{(1)}+l_{i}%
\text{ \ for each vertex }i\in V.  \label{91}
\end{equation}%
Indeed, we know by definition (\ref{lu}) that
\begin{equation}
(\mathrm{id}-M_{x}(1))(l_{i})_{i\in V}^{T}=\left( \sum_{k=1}^{d_{i}-1}\xi
_{i}^{(k)}\right) _{i\in V}^{T},  \label{lll}
\end{equation}%
from which, by definitions (\ref{mx}) and (\ref{mij}),%
\begin{equation}
\sum_{k=1}^{d_{i}-1}\xi
_{i}^{(k)}=l_{i}-\sum_{j=1}^{N}M_{ij}(1)l_{j}=l_{i}-\sum_{j=1}^{N}\left(
\sum_{e_{i}(k)\in E_{ij}}|x_{i}^{(k)}|\right)
l_{j}=l_{i}-\sum_{k=1}^{d_{i}}|x_{i}^{(k)}|l_{\omega (e_{i}(k))},
\label{91-1}
\end{equation}%
so that%
\begin{equation}
l_{i}=\sum_{k=1}^{d_{i}-1}\xi
_{i}^{(k)}+\sum_{k=1}^{d_{i}}|x_{i}^{(k)}|l_{\omega (e_{i}(k))}\text{ for
each vertex }i\in V.  \label{LI}
\end{equation}%
Combining this with (\ref{99}),%
\begin{eqnarray*}
l_{i} &=&\sum_{k=1}^{d_{i}-1}\xi
_{i}^{(k)}+\sum_{k=1}^{d_{i}}|x_{i}^{(k)}|l_{\omega
(e_{i}(k))}=\sum_{k=1}^{d_{i}-1}(\xi _{i}^{(k)}+|x_{i}^{(k)}|l_{\omega
(e_{i}(k))})+|x_{i}^{(d_{i})}|l_{\omega (e_{i}(d_{i}))} \\
&=&b_{i}^{(d_{i})}-b_{i}^{(1)}+|x_{i}^{(d_{i})}|l_{\omega (e_{i}(d_{i}))},
\end{eqnarray*}%
thus showing (\ref{91}). This proves our claim.

We next show that the contracting similarity $S_{e_{i}(k)}$ associated with
the edge $e_{i}(k)$ satisfies
\begin{equation}
S_{e_{i}(k)}(I_{\omega
(e_{i}(k))})=[b_{i}^{(k)},b_{i}^{(k)}+|x_{i}^{(k)}|l_{\omega
(e_{i}(k))}]=I_{i}(k)  \label{S2}
\end{equation}%
for each vertex $i\in V$ and each $1\leq k\leq d_{i}$. This is easily seen
by looking at the two endpoints of interval $I_{\omega (e_{i}(k))}$,
depending on whether $x_{i}^{(k)}>0$ or not. Indeed, by definition (\ref{I2}) with
vertex $i$ being replaced by vertex $\omega (e_{i}(k))$,
\begin{equation*}
I_{\omega (e_{i}(k))}=[b_{\omega (e_{i}(k))}^{(1)},b_{\omega
(e_{i}(k))}^{(1)}+l_{\omega (e_{i}(k))}].
\end{equation*}%
If $x_{i}^{(k)}>0$, we have by definition (\ref{S1}) that $%
S_{e_{i}(k)}(b_{\omega (e_{i}(k))}^{(1)})=b_{i}^{(k)}$ and%
\begin{equation*}
S_{e_{i}(k)}(b_{\omega (e_{i}(k))}^{(1)}+l_{\omega
(e_{i}(k))})=x_{i}^{(k)}l_{\omega
(e_{i}(k))}+b_{i}^{(k)}=b_{i}^{(k)}+|x_{i}^{(k)}|l_{\omega (e_{i}(k))},
\end{equation*}%
from which
\begin{eqnarray}
S_{e_{i}(k)}(I_{\omega (e_{i}(k))}) &=&[S_{e_{i}(k)}(b_{\omega
(e_{i}(k))}^{(1)}),S_{e_{i}(k)}(b_{\omega (e_{i}(k))}^{(1)}+l_{\omega
(e_{i}(k))})]  \notag \\
&=&[b_{i}^{(k)},b_{i}^{(k)}+|x_{i}^{(k)}|l_{\omega (e_{i}(k))}],  \label{S9}
\end{eqnarray}%
thus showing (\ref{S2}). On the other hand, if $x_{i}^{(k)}<0$, we similarly
have that $S_{e_{i}(k)}(b_{\omega (e_{i}(k))}^{(1)}+l_{\omega
(e_{i}(k))})=b_{i}^{(k)}$ and
\begin{equation*}
S_{e_{i}(k)}(b_{\omega (e_{i}(k))}^{(1)})=b_{i}^{(k)}-x_{i}^{(k)}l_{\omega
(e_{i}(k))}=b_{i}^{(k)}+|x_{i}^{(k)}|l_{\omega (e_{i}(k))}\text{,}
\end{equation*}%
so
\begin{eqnarray*}
S_{e_{i}(k)}(I_{\omega (e_{i}(k))}) &=&[S_{e_{i}(k)}(b_{\omega
(e_{i}(k))}^{(1)}+l_{\omega (e_{i}(k))}),S_{e_{i}(k)}(b_{\omega
(e_{i}(k))}^{(1)})] \\
&=&[b_{i}^{(k)},b_{i}^{(k)}+|x_{i}^{(k)}|l_{\omega (e_{i}(k))}],
\end{eqnarray*}%
thus showing (\ref{S2}) again. Thus (\ref{S2}) is always true.

Since by definition (\ref{555})%
\begin{equation*}
b_{i}^{(k)}+|x_{i}^{(k)}|l_{\omega (e_{i}(k))}=b_{i}^{(k+1)}-\xi
_{i}^{(k)}\leq b_{i}^{(k+1)},
\end{equation*}%
we know by (\ref{S2}) that the closed intervals $\{I_{i}(k):1\leq k\leq
d_{i}\}$ are arranged in order from left to right, which together with (\ref%
{91}) implies that
\begin{eqnarray}
\mathop{\textstyle \bigcup }\limits_{k=1}^{d_{i}}\mathrm{int}(I_{i}(k)) &=&%
\mathop{\textstyle \bigcup }\limits_{k=1}^{d_{i}}\left(
b_{i}^{(k)},b_{i}^{(k)}+|x_{i}^{(k)}|l_{\omega (e_{i}(k))}\right)  \notag \\
&\subset &\left( b_{i}^{(1)},b_{i}^{(d_{i})}+|x_{i}^{(d_{i})}|l_{\omega
(e_{i}(d_{i}))}\right) =\left( b_{i}^{(1)},b_{i}^{(1)}+l_{i}\right)
\label{92}
\end{eqnarray}%
with the disjoint union.

We are now in a position to prove the assertions $(i),(ii)$.

$(i)$. We will use (\ref{S2}) and definition (\ref{lu}) to derive (\ref{C1}%
). Indeed, recall that the intervals $I_{i}$ are defined in (\ref{I2}). Note
that $l_{i}>0$ for each $i\in V$ by (\ref{L1}). As in (\ref{000}),
 for each vertex $i\in V$ we let
\begin{equation}
I_{i}^{m}\coloneqq\mathop{\textstyle \bigcup }_{\mathbf{e}\in E_{i}^{m}}S_{%
\mathbf{e}}\left( I_{\omega (\mathbf{e})}\right) \text{ for }m=1,2,\cdots ,
\label{007}
\end{equation}%
where $E_{i}^{m}$ is the set of edges of length $m$ leaving vertex $i$,
and $\omega (\mathbf{e})$ is the terminal of path $\mathbf{e}$ as before. We
show that for each vertex $i\in V$%
\begin{equation}
\min I_{i}^{m}=b_{i}^{(1)}=\min I_{i},\text{ \ }\max
I_{i}^{m}=b_{i}^{(1)}+l_{i}=\max I_{i}\text{ \ \ for }m=1,2,\cdots,
\label{008}
\end{equation}%
so that the left and right endpoints, respectively, of all the intervals $I_{i}^{m}$ are
the same.

Indeed, we know by definition (\ref{007}) that for each vertex $i\in V$%
\begin{eqnarray*}
I_{i}^{1} &=&\mathop{\textstyle \bigcup }_{\mathbf{e}\in E_{i}^{1}}S_{%
\mathbf{e}}\left( I_{\omega (\mathbf{e})}\right) =\mathop{\textstyle \bigcup
}_{k=1}^{d_{i}}S_{e_{i}(k)}\left( I_{\omega (e_{i}(k))}\right) \\
&=&\mathop{\textstyle \bigcup }%
_{k=1}^{d_{i}}[b_{i}^{(k)},b_{i}^{(k)}+|x_{i}^{(k)}|l_{\omega (e_{i}(k))}]%
\text{ \ (using (\ref{S2})),}
\end{eqnarray*}%
from which, using the fact that $b_{i}^{(k)}+|x_{i}^{(k)}|l_{\omega
(e_{i}(k))}\leq b_{i}^{(k+1)}$ by (\ref{555}), it follows that $\min
I_{i}^{1}=b_{i}^{(1)}$, and%
\begin{equation*}
\max I_{i}^{1}=b_{i}^{(d_{i})}+|x_{i}^{(d_{i})}|l_{\omega
(e_{i}(d_{i}))}=b_{i}^{(1)}+l_{i}
\end{equation*}%
by using (\ref{91}). Hence, the (\ref{008}) is true when $m=1$ by definition
(\ref{I2}) of $I_{i}$.

Assume inductively that (\ref{008}) holds for some $m\geq 1$. Since for each vertex $%
i\in V$
\begin{equation*}
I_{i}^{m+1}=\mathop{\textstyle \bigcup }_{\mathbf{e}^{\prime }\in
E_{i}^{m+1}}S_{\mathbf{e}^{\prime }}\left( I_{\omega (\mathbf{e}^{\prime
})}\right) =\mathop{\textstyle \bigcup }_{\mathbf{e}\in E_{i}^{m}}S_{\mathbf{%
e}}\left( I_{\omega (\mathbf{e})}^{1}\right)
\end{equation*}%
by using (\ref{002-2}), it follows that
\begin{eqnarray*}
\min I_{i}^{m+1} &=&\min \{S_{\mathbf{e}}(I_{\omega (\mathbf{e})}^{1}):%
\mathbf{e}\in E_{i}^{m}\} \\
&=&\min \{S_{\mathbf{e}}(I_{\omega (\mathbf{e})}):\mathbf{e}\in E_{i}^{m}\}
\\
&=&\min I_{i}^{m}=b_{i}^{(1)}.
\end{eqnarray*}%
Similarly,
\begin{equation*}
\max I_{i}^{m+1}=\max I_{i}^{m}=b_{i}^{(1)}+l_{i}.
\end{equation*}%
Therefore, the (\ref{008}) holds for all $m\geq 1$ by induction.

Since condition (\ref{001}) holds using that $I_{i}^{1}\subset
I_{i}=[b_{i}^{(1)},b_{i}^{(1)}+l_{i}]$, and we know by (\ref{003}) that $%
F_{i}=\tbigcap_{m=1}^{\infty }I_{i}^{m}$,  (\ref{008}) gives that,%
\begin{equation*}
\mathrm{conv}\,F_{i}=\mathrm{conv}\,\tbigcap_{m=1}^{\infty
}I_{i}^{m}=[b_{i}^{(1)},\ b_{i}^{(1)}+l_{i}],
\end{equation*}%
showing that (\ref{C1}) holds true.

$(ii)$. Applying (\ref{C1}) with $i$  replaced by vertex $\omega
(e_{i}(k))$, the terminal of the edge $e_{i}(k)$, %
\begin{equation}
\mathrm{conv}\,F_{\omega (e_{i}(k))}=I_{\omega (e_{i}(k))}=[b_{\omega
(e_{i}(k))}^{(1)},\ b_{\omega (e_{i}(k))}^{(1)}+l_{\omega (e_{i}(k))}],
\label{C4}
\end{equation}%
from which it follows by (\ref{S2}) that%
\begin{equation}
S_{e_{i}(k)}(\mathrm{conv}\,F_{\omega (e_{i}(k))})=S_{e_{i}(k)}(I_{\omega
(e_{i}(k))})=[b_{i}^{(k)},b_{i}^{(k)}+|x_{i}^{(k)}|l_{\omega
(e_{i}(k))}]=I_{i}(k)  \label{C5}
\end{equation}%
for $1\leq k\leq d_{i}$.

We show that $\digamma (x,\mathcal{b})$ satisfies the COSC. Taking $U_{i}=%
\mathrm{int(conv}\,F_{i})$, from (\ref{C1})
\begin{equation*}
U_{i}=\mathrm{int}(\mathrm{conv}\,F_{i})=(b_{i}^{(1)},\ b_{i}^{(1)}+l_{i})=%
\mathrm{int}(I_{i}),
\end{equation*}%
so that each open set $U_{i}$ is not empty as $l_{i}>0$. It
follows that
\begin{eqnarray*}
\mathop{\textstyle \bigcup }_{j\in V}\mathop{\textstyle \bigcup }_{e\in
E_{ij}}S_{e}\left( U_{j}\right) &=&\mathop{\textstyle \bigcup }%
_{k=1}^{d_{i}}S_{e_{i}(k)}\left( U_{\omega (e_{i}(k))}\right) =%
\mathop{\textstyle \bigcup }_{k=1}^{d_{i}}S_{e_{i}(k)}\left( \mathrm{int(}%
I_{\omega (e_{i}(k))}\right) ) \\
&=&\mathop{\textstyle \bigcup }_{k=1}^{d_{i}}\mathrm{int}\left(
S_{e_{i}(k)}(I_{\omega (e_{i}(k))})\right) =\mathop{\textstyle \bigcup }%
_{k=1}^{d_{i}}\mathrm{int}(I_{i}(k))\text{ \ (using (\ref{S2}))} \\
&\subseteq &(b_{i}^{(1)},\ b_{i}^{(1)}+l_{i})\text{ \ \ (using (\ref{92}))}
\\
&=&U_{i}
\end{eqnarray*}%
with the union disjoint. Thus $\digamma (x,\mathcal{b})$ satisfies the COSC.

For each vertex $i\in V$, the basic gaps of the attractor $F_{i}$ are the
complementary open intervals between the closed interval%
\begin{equation*}
S_{e_{i}(k)}(\mathrm{conv}\,F_{\omega
(e_{i}(k))})=[b_{i}^{(k)},b_{i}^{(k)}+|x_{i}^{(k)}|l_{\omega
(e_{i}(k))}]=I_{i}(k)\text{ \ (using (\ref{C5}))}
\end{equation*}%
and its neighbour%
\begin{equation*}
S_{e_{i}(k+1)}(\mathrm{conv}\,F_{\omega
(e_{i}(k+1))})=[b_{i}^{(k+1)},b_{i}^{(k+1)}+|x_{i}^{(k+1)}|l_{\omega
(e_{i}(k+1))}]=I_{i}(k+1)
\end{equation*}%
for $1\leq k\leq d_{i}-1$. Specifically, they are the following open intervals%
\begin{equation*}
\left\{ \left( b_{i}^{(k)}+|x_{i}^{(k)}|l_{\omega
(e_{i}(k))},b_{i}^{(k+1)}\right) \right\} _{1\leq k\leq d_{i}-1}.
\end{equation*}%
that are arranged in order from left to right, thus showing (\ref{bg2}) for
each vertex $i\in V$.

The basic gap lengths of the attractor $F_{i}$ are the lengths of the open
intervals in (\ref{bg2}), which are equal to
\begin{equation*}
b_{i}^{(k+1)}-\left( b_{i}^{(k)}+|x_{i}^{(k)}|l_{\omega (e_{i}(k))}\right)
=\xi _{i}^{(k)}\text{ \ (}1\leq k\leq d_{i}-1\text{)}
\end{equation*}%
by using definition (\ref{555}), thus showing (\ref{gl2}).

Finally, if all $\xi _{i}^{(k)}>0$, then $\digamma (x,\mathcal{b})$
satisfies the CSSC, since%
\begin{equation*}
\mathop{\textstyle \bigcup }_{v\in V}\mathop{\textstyle \bigcup }_{e\in
E_{uv}}S_{e}\left( \mathrm{conv}\,F_{v}\right) =\mathop{\textstyle \bigcup }%
_{k=1}^{d_{i}}S_{e_{i}(k)}\left( \mathrm{conv}\,F_{\omega (e_{i}(k))}\right)
=\mathop{\textstyle \bigcup }_{k=1}^{d_{i}}I_{i}(k)
\end{equation*}%
with the disjoint union, as the intervals $I_{i}(k)$ and $I_{i}(k+1)$ are
separated by distance $\xi _{i}^{(k)}$, which are strictly positive.
\end{proof}

\begin{remark}
\label{R1}\textrm{Note that any point $x$ belongs to $P$ if
\begin{equation}
\max_{i\in V}\Bigg\{ \sum_{k=1}^{d_{i}}|x_{i}^{(k)}|\Bigg\} <1.
\label{large}
\end{equation}%
This is because
\begin{equation}
r_{\sigma }(M_{x}(1))\leq \max_{i\in V}\Bigg\{
\sum_{j=1}^{N}\sum\limits_{e_{i}(k)\in E_{ij}}|x_{i}^{(k)}|\Bigg\}
=\max_{i\in V}\Bigg\{ \sum_{k=1}^{d_{i}}|x_{i}^{(k)}|\Bigg\} <1,  \label{SE}
\end{equation}%
using the elementary fact that the spectral radius of a nonnegative
matrix is no greater than any row sum, see for example \cite[Equation (1.9)]%
{Seneta}. Therefore, every $x\in P_{0}$ satisfying (\ref{large})
belongs to $P$, and all the assertions $(i)$, $(ii)$ in Lemma \ref{P1} hold
true, provided that $(l_{i})_{i\in V}$ are chosen as in (\ref{lu}). }
\end{remark}

We now look at subsets $P$, depending on a number $\delta >0$%
, which will give rise to a special class of GD-IFSs, satisfying the CSSC,
having attractors $\{F_{i}\}_{i\in V}$ with the property that $\mathrm{conv}%
\,F_{i}=[0,1]$, and all the basic gaps of $F_{i}$ have the same length $%
\delta $.

\begin{definition}
\label{Def2}Let $\delta $ be a small number such that%
\begin{equation}
0<\delta <\min_{i\in V}\Bigg\{ \frac{1}{d_{i}-1}\Bigg\}  \label{x1}
\end{equation}%
(recall our assumption that the out-degree $d_i$ at vertex $i$ satisfies $d_i\geq 2$
for all $i$).
We define a set $\mathcal{A}(\delta )$ by%
\begin{equation}
\begin{array}{cc}
\mathcal{A}(\delta )\coloneqq\Big\{(x_{1}^{(1)},\cdots
,x_{1}^{(d_{1})},x_{2}^{(1)},\cdots ,x_{2}^{(d_{2})},\cdots
,x_{N}^{(1)},\cdots ,x_{N}^{(d_{N})},\delta ,\cdots ,\delta )\in
%TCIMACRO{\U{211d} }%
%BeginExpansion
\mathbb{R}
%EndExpansion
^{n}:|x_{i}^{(k)}|>0 &  \\
\text{ }\hskip0.1cm\text{ and }|x_{i}^{(1)}|+\cdots
+|x_{i}^{(d_{i})}|=1-\left( d_{i}-1\right) \delta \text{ for all }i\in
V,1\leq k\leq d_{i}\Big\}. &
\end{array}
\label{AA1}
\end{equation}%
%where $n$ is given by (\ref{nn}).
\end{definition}

Let $M_{x}(1)$ be an $N\times N$ matrix associated with point $x$ as in (\ref%
{mx}) for $s=1$. For each $x \in \mathcal{A}(\delta )$, the spectral
radius of matrix $M_{x}(1)$ is less than $1$, since
\begin{equation}
\max_{i\in V}\Bigg\{ \sum_{k=1}^{d_{i}}|x_{i}^{(k)}|\Bigg\} =\max_{i\in
V}\left\{ 1-(d_{i}-1)\delta \right\} <1\text{ (using (\ref{x1}))}
\label{xx7}
\end{equation}%
and hence,
\begin{equation}
\mathcal{A}(\delta )\subset P  \label{xx8}
\end{equation}%
where the set $P$ is as in (\ref{sm}). Moreover,%
\begin{equation}
(\mathrm{id}-M_{x}(1))\left(
\begin{array}{c}
1 \\
1 \\
\vdots \\
1%
\end{array}%
\right) :=\left(
\begin{array}{c}
(d_{1}-1)\delta \\
(d_{2}-1)\delta \\
\vdots \\
(d_{N}-1)\delta%
\end{array}%
\right)  \label{xx3}
\end{equation}%
so that (\ref{lu}) is satisfied with%
\begin{equation}
l_{i}=1\text{ \ and \ }\xi _{i}^{(k)}=\delta \text{ \ for }i\in V\text{; }%
1\leq k\leq d_{i}-1\text{,}  \label{xx9}
\end{equation}%
this is because for each $i\in V$, by definitions (\ref{AA1}) and (\ref{mij}%
),%
\begin{eqnarray}
(d_{i}-1)\delta &=&1-\left( |x_{i}^{(1)}|+\cdots +|x_{i}^{(d_{i})}|\right)
\label{xx11} \\
&=&1-\sum_{j=1}^{N}\Bigg( \sum_{e_{i}(k)\in E_{ij}}|x_{i}^{(k)}|\Bigg)
=\sum_{j=1}^{N}\left( \mathrm{id}-M_{x}(1)\right) _{ij}\left(
\begin{array}{c}
1 \\
1 \\
\vdots \\
1%
\end{array}%
\right) .  \notag
\end{eqnarray}

Let $\{b_{i}^{(k)}\}_{i\in V,1\leq k\leq d_{i}}$ be a family of real numbers
given by%
\begin{equation}
b_{i}^{(1)}=0\text{ \ and \ }b_{i}^{(k+1)}=b_{i}^{(k)}+|x_{i}^{(k)}|+\delta
\text{ \ for }i\in V,1\leq k\leq d_{i}-1\text{,}  \label{bb4}
\end{equation}%
so that%
\begin{equation}
b_{i}^{(k+1)}=|x_{i}^{(1)}|+|x_{i}^{(2)}|+\cdots +|x_{i}^{(k)}|+k\delta
\text{ \ (}i\in V,1\leq k\leq d_{i}-1\text{). }  \label{bb5}
\end{equation}%
Clearly, each $b_{i}^{(k+1)}\in (0,1)$ for $1\leq k\leq d_{i}-1$\ by using (%
\ref{xx11}).

Let $\mathcal{b}_{0}$, $\ell _{1}$ be two vectors defined by
\begin{eqnarray}
\mathcal{b}_{0} &\coloneqq&(b_{i}^{(1)},b_{2}^{(1)},\cdots
,b_{N}^{(1)})=(0,0,\cdots ,0),  \notag \\
\ell _{1} &\coloneqq&(l_{1},l_{2},\cdots ,l_{N})=(1,1,\cdots ,1).
\label{LL2-1}
\end{eqnarray}%
In this situation, for  $x \in\mathcal{A}(\delta )$, the
contracting similarities defined in (\ref{S1}) read%
\begin{equation}
S_{e_{i}(k)}(t)=x_{i}^{(k)}t+b_{i}^{(k)}-x_{i}^{(k)}\mathbf{1}%
_{\{x_{i}^{(k)}<0\}}\text{ \ for a variable }t\in
%TCIMACRO{\U{211d} }%
%BeginExpansion
\mathbb{R}
%EndExpansion
\label{S5}
\end{equation}%
for $i\in V,1\leq k\leq d_{i}$, which will give arise to a GD-IFS satisfying
the CSSC. This will be used in Theorem \ref{T2} below.

\begin{corollary}
\label{C2}Let $G=(V,E)$ be a digraph with $d_{i}\geq 2$ for $i\in
V=\{1,2,\cdots ,N\}$. Let $\delta $ satisfy $(\ref{x1})$. For
$x \in \mathcal{A}(\delta )$, let%
\begin{equation*}
\digamma (x):=\{S_{e_{i}(k)}:i\in V,1\leq k\leq d_{i}\}
\end{equation*}%
be a GD-IFS given as in $(\ref{S5})$, with attractors $(F_{i})_{i\in V}$. Then
the following statements hold.

\begin{enumerate}
\item[$(i).$] For each vertex $i\in V$, $\mathrm{conv}\,F_{i}=[0,1].$

\item[$(ii).$] For each vertex $i\in V$,
\begin{equation}
S_{e_{i}(1)}([0,1])=\left[ 0,|x_{i}^{(1)}|\right]  \label{61}
\end{equation}%
so that $|x_{i}^{(1)}|\in F_{i}$. The basic gaps of the attractor $%
F_{i}$ are given by%
\begin{equation}
\left( |x_{i}^{(1)}|+\cdots +|x_{i}^{(k)}|+(k-1)\delta ,|x_{i}^{(1)}|+\cdots
+|x_{i}^{(k)}|+k\delta \right)  \label{62}
\end{equation}%
for every $1\leq k\leq d_{i}-1$, so that the basic gap lengths are all equal
to the same number, $\delta $ say. Moreover, the  GD-IFS $\digamma (x)$ satisfies the
CSSC.
\end{enumerate}
\end{corollary}

\begin{proof}
Let $x\in \mathcal{A}(\delta )$. Then $x\in P$ by using (\ref%
{xx8}), and condition (\ref{lu}) is also satisfied by  (\ref{xx3}).
Thus all the assumptions in Lemma \ref{P1} are satisfied. Applying Lemma \ref%
{P1}$(i)$ and using (\ref{I2}) with $b_{i}^{(1)}=0$ and $l_{i}=1$,%
\begin{equation*}
\mathrm{conv}\,F_{i}=[b_{i}^{(1)},b_{i}^{(1)}+l_{i}]=[0,1],
\end{equation*}%
thus showing $(i)$.

To show $(ii)$, noting that $I_{\omega (e_{i}(k))}=[0,1]$ and $%
l_{\omega (e_{i}(k))}=1,b_{i}^{(1)}=0$, we know by (\ref{S2}) that%
\begin{equation*}
S_{e_{i}(1)}([0,1])=S_{e_{i}(1)}(I_{\omega
(e_{i}(k))})=[b_{i}^{(1)},b_{i}^{(1)}+|x_{i}^{(1)}|l_{\omega
(e_{i}(1))}]=[0,|x_{i}^{(1)}|]
\end{equation*}%
thus showing (\ref{61}).

By (\ref{bg2}), (\ref{bb5}), the basic gaps of the attractor $F_{i}$ are
given by%
\begin{eqnarray*}
\left( b_{i}^{(k)}+|x_{i}^{(k)}|l_{\omega (e_{i}(k))},b_{i}^{(k+1)}\right)
&=&\left( b_{i}^{(k)}+|x_{i}^{(k)}|,b_{i}^{(k+1)}\right) \\
&=&\left( |x_{i}^{(1)}|+\cdots +|x_{i}^{(k)}|+(k-1)\delta
,|x_{i}^{(1)}|+\cdots +|x_{i}^{(k)}|+k\delta \right)
\end{eqnarray*}%
for every $1\leq k\leq d_{i}-1$, thus showing (\ref{62}). From this, it is
clear that the basic gap lengths all are equal to the same number which
we call $\delta $.
Finally, $\digamma (x)$ satisfies the CSSC by Lemma \ref{P1}$(ii)$
since all $\xi _{i}^{(k)}=\delta >0$.
\end{proof}

\section{Criteria for graph directed attractors not to be self-similar sets}

\label{Sect4}In this section we give some sufficient conditions under which
GD-attractors cannot be realised as attractors of any standard IFSs with or
without the COSC.

For a directed path $L$, let $A(L)$ (resp. $A(L^{c})$) be the set of the
absolute values of the contraction ratios of the similarities associated
with the edges in $L$ (resp. not in $L$). Recall the definition of $\Lambda _{u}$
from (\ref{14}).

\begin{lemma}
\label{main thm} Assume that $(V,E)$ is a digraph with $d_{w}\geq 2$ for all
$w\in V$ and $L$ is a directed circuit that does not go through every vertex
in $V$. Let $u$ be a vertex outside $L$ and $v$ a vertex in $L$, assume that
there exists a directed path from $u$ to $v$. Consider a COSC GD-IFS based
on this digraph. With the notation above, suppose that the following three
conditions hold:

\begin{enumerate}
\item[$(i)$] $(A(L))^{\mathbb{Q}^{\ast }}\cap (A(L^{c}))^{\mathbb{Q}%
_{+}^{\ast }}=\emptyset .$

\item[$(ii)$] $\Lambda _{u}\neq \emptyset $ and $\Lambda _{v}\neq \emptyset $%
.

\item[$(iii)$] For all pairs $(w,k)\neq (z,m)$ with $\lambda _{z}^{(m)}\neq
0 $ where $w,z\in V$ and $1\leq k\leq d_{w}-1,1\leq m\leq d_{z}-1$,%
\begin{equation*}
\lambda _{w}^{(k)}/\lambda _{z}^{(m)}\notin (A(L)\cup A(L^{c}))^{\mathbb{Q}}.
\end{equation*}
\end{enumerate}

\noindent Then the graph-directed IFS attractor $F_{u}$ is not the attractor
of any COSC standard IFS.
\end{lemma}

Basically, condition $(i)$ means that linear combinations of numbers $\{\log
|\rho _{e}|:e\in A(L)\}$ over $\mathbb{Q}^{\ast }$, that is,%
\begin{equation*}
\sum_{e\in A(L)}q_{e}\log |\rho _{e}|\text{ \ for }(q_{e})_{e\in A(L)}\in (%
\mathbb{Q}^{\#A(L)})^{\ast }
\end{equation*}%
where $(\mathbb{Q}^{\#A(L)})^{\ast }$ is the set of non-zero vectors in $%
\mathbb{Q}^{\#A(L)}$ as before, are different from those of numbers $\{\log
|\rho _{e}|:e\in A(L^{c})\}$ over $\mathbb{Q}_{+}^{\ast }$, while condition $%
(ii)$ means that not all basic gaps associated with $u$ and $v$ are empty,
and condition $(iii)$ means that $\log (\lambda _{w}^{(k)}/\lambda
_{z}^{(m)})$ for all distinct basic gaps of positive lengths are different
from linear combinations of numbers $\{\log |\rho _{e}|:e\in E\}$ over $%
\mathbb{Q}$. Note that condition $(i)$ requires a certain homogeneity, on
the ratios of the gap length set of a COSC self-similar GD-attractor, which
does not necessarily hold when $(ii)$ and $(iii)$ are satisfied. Note that
among the three conditions $(i)$, $(ii)$, $(iii)$, no two of them
imply the third.

\begin{proof}
We show that the strict dichotomy required by Lemma \ref{Lemma 2.9} $(ii)$
for a graph-directed attractor fails for $F_{u}$ satisfying the conditions
of this theorem.

Let $u$ be a vertex outside $L$ and $v$ a vertex in $L$. For any $w\neq u$
in $V$, let
\begin{equation*}
R(uw)=\{|\rho _{\mathbf{e}}|:\ \mathbf{e}\text{ is a directed path from }u%
\text{ to }w\},
\end{equation*}%
and let%
\begin{equation*}
R(uu)=\{1\}\cup \{|\rho _{\mathbf{e}}|:\ \mathbf{e}\text{ is a directed
circuit containing }u\}.
\end{equation*}%
With the above notation, the union (\ref{doubleunion}) becomes
\begin{equation}
\Theta \coloneqq\mathrm{GL}(F_{u})=\mathop{\textstyle \bigcup }_{w\in
V}\Lambda _{w}R(uw)=\mathop{\textstyle \bigcup }_{w\in V}\mathop{\textstyle
\bigcup }_{\lambda \in \Lambda _{w}}\lambda R(uw).  \label{newdoubleunion}
\end{equation}%
By condition $(ii)$, we can choose two non-zero basic gap lengths $\lambda
_{u}\in \Lambda _{u}$, $\lambda _{v}\in \Lambda _{v}$. Since there exists a
directed path $\mathbf{e}$ from $u$ to $v$, we can choose a number
\begin{equation*}
\theta \coloneqq\lambda _{v}|\rho _{\mathbf{e}}|\in \lambda _{v}R(uv)\subset
\mathrm{GL}(F_{u}).
\end{equation*}

Recall that $\rho _{L}$ denotes the product of the contraction ratios on the
edges of $L$. For each integer $k\geq 0$, we define $\mathbf{e}L^{k}$ by $%
\mathbf{e}L^{0}:=\mathbf{e}$ and
\begin{equation*}
\mathbf{e}L^{k}:=\mathbf{e}\underset{k\text{ times}}{\underbrace{L\cdots L}}%
\text{ \ for }k\geq 1,
\end{equation*}%
all of which are directed paths from $u$ to $v$, so that $|\rho _{\mathbf{e}%
L^{k}}|\in R(uv)$. Note that%
\begin{equation}
|\rho _{L}|\in R_{\mathrm{GL}(F_{u})}(\theta ),  \label{41}
\end{equation}%
since for every $k\geq 0$,
\begin{equation*}
\theta |\rho _{L}|^{k}=\lambda _{v}|\rho _{\mathbf{e}}||\rho
_{L}|^{k}=\lambda _{v}|\rho _{\mathbf{e}L^{k}}|\in \lambda _{v}R(uv)\subset
\mathrm{GL}(F_{u})
\end{equation*}%
by using (\ref{newdoubleunion}), which implies (\ref{41}) by definition (\ref{R})
with $\theta ^{\prime }$ being replaced by $\theta \in \mathrm{GL}(F_{u})$.

Set $A_{1}\coloneqq A(L),\ A_{2}\coloneqq A(L^{c})$. Since $|\rho _{L}|\in
(A(L))^{\mathbb{Z}_{+}^{\ast }}=A_{1}^{\mathbb{Z}_{+}^{\ast }}\subset A_{1}^{%
\mathbb{Q}_{+}^{\ast }}$, we obtain by (\ref{41})
\begin{equation}
R_{\mathrm{GL}(F_{u})}(\theta )\cap A_{1}^{\mathbb{Q}_{+}^{\ast }}\neq
\emptyset .  \label{neq}
\end{equation}

Let
\begin{equation*}
r\in R_{\mathrm{GL}(F_{u})}(\lambda _{u}).
\end{equation*}%
By definition (\ref{R}), there exists a geometric sequence $\{\theta
^{\prime }r^{k}\}_{k=0}^{\infty }\subset \mathrm{GL}(F_{u})$ containing $%
\lambda _{u}$ with $\theta ^{\prime }\in \mathrm{GL}(F_{u})$. Note that $%
\lambda _{u}\in \lambda _{u}R(uu)\subset \mathrm{GL}(F_{u})$ by (\ref%
{newdoubleunion}).

We claim that%
\begin{equation}
\theta ^{\prime }=\lambda _{u}.  \label{42}
\end{equation}%
To see this, taking the decomposition of $\Theta =\mathrm{GL}(F_{u})$ given
by (\ref{newdoubleunion}), the requirements for Lemma \ref{GDR} $(ii)$, with
$\lambda _{j}$ varying in $\{\lambda \in \Lambda _{w}:w\in V\}$, $A_{j}$
varying in $\{R_{uw}:w\in V\}$ and with $A=A(L)\mathop{\textstyle \bigcup }
A(L^{c})$, are satisfied by assumption $(iii)$. Thus, there is a unique $%
w\in V$ and a unique $\lambda \in \Lambda _{w}$ such that
\begin{equation}
\{\theta ^{\prime }r^{k}\}_{k=0}^{\infty }\subset \lambda R(uw),  \label{44}
\end{equation}%
and
\begin{equation}
\theta ^{\prime }r^{k}\notin \lambda ^{\prime }R(uz)\text{ \ for all }%
(\lambda ^{\prime },z)\neq (\lambda ,w)\text{ and all }k\geq 0  \label{43}
\end{equation}%
by (\ref{51}). Thus, $\lambda _{u}\in \{\theta ^{\prime
}r^{k}\}_{k=0}^{\infty }\subset \lambda R(uw)$. On the other hand, noting
that $1\in R(uu)$ so that%
\begin{equation}
\lambda _{u}\in \lambda _{u}R(uu),  \label{ghw}
\end{equation}%
we conclude that $\lambda =\lambda _{u},w=u$ by (\ref{43}).

Since $r<1,$ we have $\lambda _{u}=\theta ^{\prime }r^{k}\leq \theta
^{\prime }$ for some $k$. As $R(uu)\subset (0,1]$, we know by (\ref{44})
that
\begin{equation}
\theta ^{\prime }r^{k}\in \lambda R(uw)=\lambda _{u}R(uu)  \label{45}
\end{equation}%
for every $k\geq 0$, which gives that $\theta ^{\prime }\leq \lambda _{u}$
on taking $k=0$, and so $\lambda _{u}=\theta ^{\prime }$, thus proving our
claim (\ref{42}).

By (\ref{42}) and (\ref{45}) with $k=1$,
\begin{equation*}
\lambda _{u}r=\theta ^{\prime }r\in \lambda _{u}R(uu),
\end{equation*}%
from which we see that $r\in R(uu)$, thus showing that
\begin{equation}
R_{\mathrm{GL}(F_{u})}(\lambda _{u})\subset R(uu),  \label{47}
\end{equation}%
since $r$ is any number in $R_{\mathrm{GL}(F_{u})}(\lambda _{u})$.

On the other hand, since $u$ is not in the circuit $L$, any directed circuit
$L^{\prime }$ containing $u$ must also visit some edge outside $L$ as well,
implying that $|\rho _{L^{\prime }}|\in (A(L))^{\mathbb{Z}_{+}}(A(L^{c}))^{%
\mathbb{Z}_{+}^{\ast }}=A_{1}^{\mathbb{Z}_{+}}A_{2}^{\mathbb{Z}_{+}^{\ast }}$
and
\begin{eqnarray}
R(uu) &=&\{1\}\cup \{|\rho _{L^{\prime }}|:\ L^{\prime }\text{ is a directed
circuit containing }u\}  \notag \\
&\subset &\{1\}\cup A_{1}^{\mathbb{Z}_{+}}A_{2}^{\mathbb{Z}_{+}^{\ast
}}\subset \{1\}\cup A_{1}^{\mathbb{Q}_{+}}A_{2}^{\mathbb{Q}_{+}^{\ast }}.
\label{48}
\end{eqnarray}%
Noting that by assumption $(i)$
\begin{equation*}
A_{1}^{\mathbb{Q}^{\ast }}\cap A_{2}^{\mathbb{Q}_{+}^{\ast }}=(A(L))^{%
\mathbb{Q}^{\ast }}\cap (A(L^{c}))^{\mathbb{Q}_{+}^{\ast }}=\emptyset
\end{equation*}%
so that $A_{1}^{\mathbb{Q}_{+}}A_{2}^{\mathbb{Q}_{+}^{\ast }}\cap A_{1}^{%
\mathbb{Q}_{+}^{\ast }}=\emptyset $ by (\ref{2.15}), it follows that
\begin{eqnarray}
R_{\mathrm{GL}(F_{u})}(\lambda _{u})\cap A_{1}^{\mathbb{Q}_{+}^{\ast }}
&\subset &R(uu)\cap A_{1}^{\mathbb{Q}_{+}^{\ast }}\text{ \ (using (\ref%
{47}))}  \notag \\
&\subset &\left( \{1\}\cup A_{1}^{\mathbb{Q}_{+}}A_{2}^{\mathbb{Q}_{+}^{\ast
}}\right) \cap A_{1}^{\mathbb{Q}_{+}^{\ast }}\text{ \ (using (\ref{48}))}
\notag \\
&=&\left( \{1\}\cap A_{1}^{\mathbb{Q}_{+}^{\ast }}\right) \cup \left( A_{1}^{%
\mathbb{Q}_{+}}A_{2}^{\mathbb{Q}_{+}^{\ast }}\cap A_{1}^{\mathbb{Q}%
_{+}^{\ast }}\right) =\emptyset  \label{eq}
\end{eqnarray}%
using that $\{1\}\cap A_{1}^{\mathbb{Q}_{+}^{\ast }}=\emptyset $,
since all numbers in $A_{1}^{\mathbb{Q}_{+}^{\ast }}$ are strictly less than $%
1 $.

Finally, since (\ref{neq}) and (\ref{eq}) hold simultaneously, Lemma \ref%
{Lemma 2.9}$(ii)$ implies that $F_{u}$ cannot be the attractor of any COSC
standard IFS.
\end{proof}

Note that the assumption `there exists a directed path from $u$ to $v$' in
Lemma \ref{main thm} is necessary. The following example shows that without
this assumption, the GD-attractor may be an attractor of some standard IFS
(with or without the COSC).

\begin{example}
\label{EX3}Let $G=(V,E)$ be the digraph (not strongly connected) in Figure \ref%
{Fig3} with $V=\{1,2,3\}$ and $E$ consisting of seven edges, three of which
leave vertex $1$ (including one loop).
\begin{figure}[h]
\begin{center}
\includegraphics[width=8cm, height=6cm]{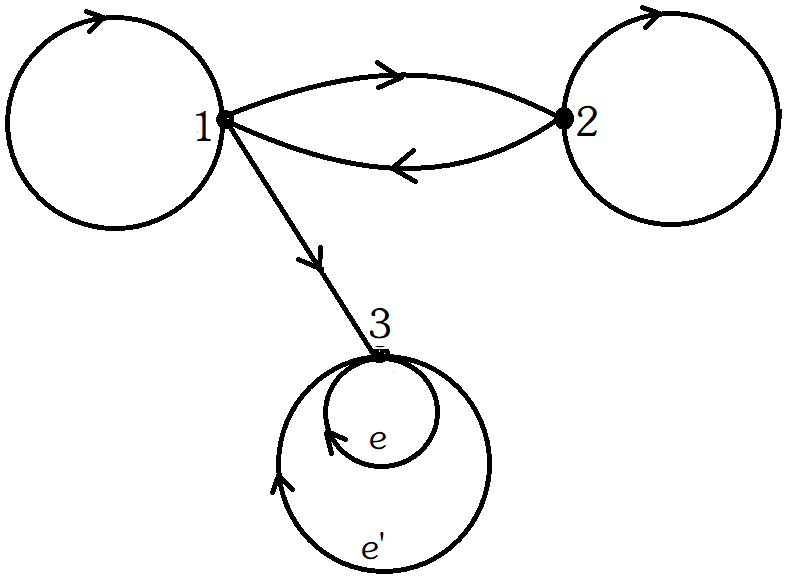}
\end{center}
\caption{{~}$F_{3}$ is an attractor of a standard IFS.}
\label{Fig3}
\end{figure}
Let $\{S_{e}\}_{e\in E}$ be any COSC GD-IFS, having GD-attractors $%
F_{1},F_{2},F_{3}$ associated with vertices $1,2,3$ respectively. By (\ref%
{gdattract}), the set $F_{3}$ satisfies
\begin{equation*}
F_{3}=S_{e}(F_{3})\cup S_{e^{\prime }}(F_{3}),
\end{equation*}%
which is an attractor of the standard IFS $\{S_{e},S_{e^{\prime }}\}$. Note
that there is no directed path from vertex $3$ to other two vertices $1,2$.
\end{example}

We give an example to illustrate Lemma \ref{main thm}. Our example is a
digraph that has three vertices and is not strongly connected.

\begin{example}[Three-vertex digraph]
\label{EX2}Let $G=(V,E)$ be the (not strongly) connected digraph in Fig. \ref%
{Fig2} with $V=\{1,2,3\}$ and $E$ consisting of seven edges.
\begin{figure}[h]
\begin{center}
\includegraphics[width=8cm, height=7cm]{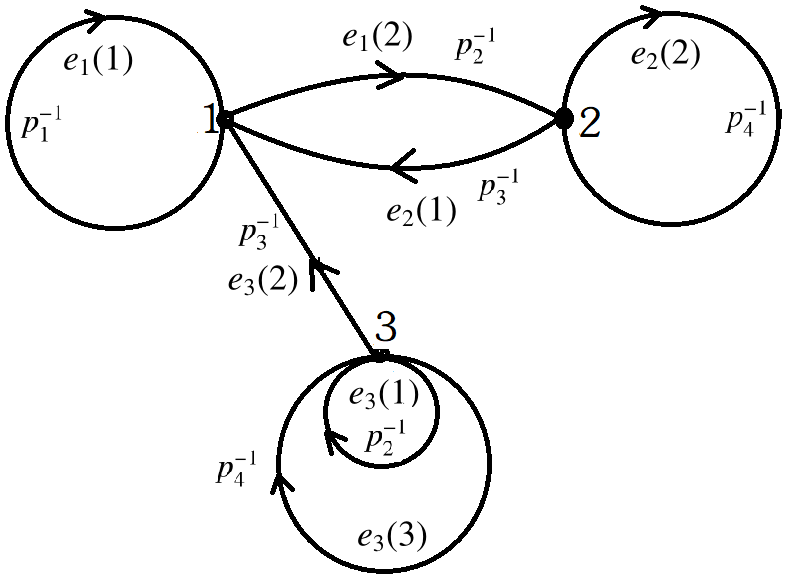}
\end{center}
\caption{${F}_{3}$ is not an attractor of any standard COSC IFS.}
\label{Fig2}
\end{figure}
Note that the out-degrees of the vertices are respectively%
\begin{equation*}
d_{1}=d_{2}=2,d_{3}=3.
\end{equation*}%
Let $L=e_{1}(1)$ be a loop (circuit) so that vertex $u=3$ is outside $L$
whilst vertex $v=1$ is inside $L$. A directed path from $u$ to $v$ is
labelled by $e_{3}(2)$.

Let $x^{\prime }$ be a point given by
\begin{equation*}
x^{\prime
}=(p_{1}^{-1},p_{2}^{-1},p_{3}^{-1},p_{4}^{-1},p_{2}^{-1},p_{3}^{-1},p_{4}^{-1},\lambda ,p_{5}\lambda ,0,\pi \lambda ),
\end{equation*}%
where $\{p_{j}\}_{1\leq j\leq 5}$ are five distinct positive prime numbers.
The matrix $M_{x^{\prime }}(1)$ defined by (\ref{mx}) is given by%
\begin{equation*}
M:=M_{x^{\prime }}(1)=\left(
\begin{array}{ccc}
p_{1}^{-1} & p_{2}^{-1} & 0 \\
p_{3}^{-1} & p_{4}^{-1} & 0 \\
p_{3}^{-1} & 0 & p_{2}^{-1}+p_{4}^{-1}%
\end{array}%
\right) .
\end{equation*}%
The point $x^{\prime }$ belongs to the set $P$ in (\ref{sm}) by using (\ref%
{large}) as the sum of each row of matrix $M$ is bounded by $1$, that is,
\begin{equation*}
\max
\{p_{1}^{-1}+p_{2}^{-1}+0,p_{3}^{-1}+p_{4}^{-1}+0,p_{3}^{-1}+0+p_{2}^{-1}+p_{4}^{-1}\}<1.
\end{equation*}

Let $\ell =(l_{1},l_{2},l_{3})$ be determined by (\ref{lu}), that is,%
\begin{equation}
\ell ^{T}=\left(
\begin{array}{c}
l_{1} \\
l_{2} \\
l_{3}%
\end{array}%
\right) =(\mathrm{id}-M)^{-1}\left(
\begin{array}{c}
\lambda \\
p_{5}\lambda \\
0+\pi \lambda%
\end{array}%
\right) .  \label{L2}
\end{equation}

Let $\mathcal{b}=(0,0,0)$ and let $\digamma (x^{\prime }):=\digamma
(x^{\prime },\mathcal{b})$ be a GD-IFS constructed as in Lemma \ref{P1},
which is given by%
\begin{eqnarray*}
S_{e_{1}(1)}(t) &=&p_{1}^{-1}t,\text{ \ }%
S_{e_{1}(2)}(t)=p_{2}^{-1}t+b_{1}^{(2)}\text{,} \\
S_{e_{2}(1)}(t) &=&p_{3}^{-1}t,\text{ \ }%
S_{e_{2}(2)}(t)=p_{4}^{-1}t+b_{2}^{(2)}\text{,} \\
S_{e_{3}(1)}(t) &=&p_{2}^{-1}t,\text{ \ }%
S_{e_{3}(2)}(t)=p_{3}^{-1}t+b_{3}^{(2)}\text{, \ }%
S_{e_{3}(3)}(t)=p_{4}^{-1}t+b_{3}^{(3)}\text{ \ for }t\in
%TCIMACRO{\U{211d} }%
%BeginExpansion
\mathbb{R}
%EndExpansion
,
\end{eqnarray*}%
where $\{b_{1}^{(2)},b_{2}^{(2)},b_{3}^{(2)},b_{3}^{(3)}\}$ are determined
by (\ref{555}), with $\ell =(l_{1},l_{2},l_{3})$ determined by (\ref{L2}).

By Lemma \ref{P1}, such a GD-IFS, $\digamma (x^{\prime })$, satisfies the
COSC, and the basic gap length sets at three vertices are respectively%
\begin{eqnarray}
\lambda _{1}^{(1)} &=&\lambda \text{ \ (at vertex }1\text{),}  \notag \\
\lambda _{2}^{(1)} &=&p_{5}\lambda \text{ \ (at vertex }2\text{),}  \notag \\
\text{ }\lambda _{3}^{(1)} &=&0,\lambda _{3}^{(2)}=\pi \lambda \text{ \ (at
vertex }3\text{),}  \label{L4}
\end{eqnarray}%
so that the sets of positive gap lengths at the vertices are
given by
\begin{eqnarray}
\Lambda _{1} &=&\{\lambda \}\text{ \ (at vertex }1\text{),}  \notag \\
\Lambda _{2} &=&\{p_{5}\lambda \}\text{ \ (at vertex }2\text{),}  \notag \\
\Lambda _{3} &=&\{\pi \lambda \}\text{ \ (at vertex }3\text{).}  \label{L5}
\end{eqnarray}

Since $L=e_{1}(1)$ is a loop, we see that%
\begin{equation}
A(L)=\{p_{1}^{-1}\}\text{ \ \ and \ \ }A(L^{c})=%
\{p_{2}^{-1},p_{3}^{-1},p_{4}^{-1}\},  \label{L6}
\end{equation}%
so that the contraction ratio set $A$ is given by%
\begin{equation}
A=A(L)\cup A(L^{c})=\{p_{1}^{-1},p_{2}^{-1},p_{3}^{-1},p_{4}^{-1}\}.
\label{L7}
\end{equation}

We show that conditions $(i)$, $(ii)$, $(iii)$ in Lemma \ref{main thm} are
all satisfied when $L=e_{1}(1)$, $u=3$ and $v=1$. Thus the attractor $F_{3}$
of the GD-IFS, $\digamma (x^{\prime })$ above, is not the attractor of any
COSC standard IFS.

To verify condition $(i)$, we need to show that%
\begin{equation*}
(A(L))^{\mathbb{Q}^{\ast }}\cap (A(L^{c}))^{\mathbb{Q}_{+}^{\ast
}}=\emptyset ,
\end{equation*}%
where $A(L)$, $A(L^{c})$ are given as in (\ref{L6}). Otherwise, there would
exist some non-zero rational number $q$ such that%
\begin{equation*}
\left( p_{1}^{-1}\right) ^{q}\in \{p_{2}^{-1},p_{3}^{-1},p_{4}^{-1}\}^{%
\mathbb{Q}_{+}^{\ast }},
\end{equation*}%
which would imply%
\begin{equation*}
1\in \{p_{1}^{-1},p_{2}^{-1},p_{3}^{-1},p_{4}^{-1}\}^{\mathbb{Q}_{+}^{\ast }}
\end{equation*}%
a contradiction by using Proposition \ref{easy} in the Appendix.

Condition $(ii)$ holds by directly using (\ref{L5}).

Finally, for condition $(iii)$, we know from (\ref{L4}) that all the ratios $%
\lambda _{w}^{(k)}/\lambda _{z}^{(m)}$ of basic gap lengths for $(w,k)\neq
(z,m)$ lie in the following set%
\begin{equation*}
\left\{ \frac{1}{p_{5}},\frac{1}{\pi },p_{5},\frac{p_{5}}{\pi },0,\pi ,\frac{%
\pi }{p_{5}}\right\} ,
\end{equation*}%
each number in which does not belong to $A^{\mathbb{Q}}=%
\{p_{1}^{-1},p_{2}^{-1},p_{3}^{-1},p_{4}^{-1}\}^{\mathbb{Q}}$ by using
Proposition \ref{easy} in Appendix and the fact that $\pi $ is
transcendental. Thus, condition $(iii)$ is satisfied.
\end{example}

We mention in passing that one can also construct a GD-IFS with the \emph{%
CSSC}, whose GD-attractor is not attractor of any standard IFS. For example,
let $p_{6}$ be a prime different from other $p_{j}\ (1\leq j\leq 5),$ and
let
\begin{equation*}
x^{\prime \prime
}=(p_{1}^{-1},p_{2}^{-1},p_{3}^{-1},p_{4}^{-1},p_{2}^{-1},p_{3}^{-1},p_{4}^{-1},\lambda ,p_{5}\lambda ,p_{6}\lambda ,\pi \lambda ).
\end{equation*}%
Such a point $x^{\prime \prime }$ also belongs to the set $P$, and the
corresponding GD-IFS, $\digamma (x^{\prime \prime })$ associated with $%
x^{\prime \prime }$ in a way of Lemma \ref{P1}, satisfies the CSSC. When $%
L=e_{1}(1)$, $u=3$ and $v=1$, the attractor $F_{3}$ of this $\digamma
(x^{\prime \prime })$ is not an attractor of any standard IFS. We omit the
details.

\begin{lemma}
\label{main}For a strongly connected digraph $G=(V,E)$ with $d_{w}\geq 2$
for all $w\in V$, let $A$ be the absolute contraction ratio set of a COSC
GD-IFS based on $G$, having $\left( F_{u}\right) _{u\in V}$ as its attractors%
\emph{. }Suppose that the following conditions hold:

\begin{enumerate}
\item[$(i^{\prime })$] All the contraction ratios have different absolute
values, and $1\notin A^{\mathbb{Q}^{\ast }}$.

\item[$(ii^{\prime })$] $\Lambda _{w}\neq \emptyset $ for all $w\in V$.

\item[$(iii)$] For all pairs $(w,k)\neq (z,m)$ with $\lambda _{z}^{(m)}\neq
0 $ where $w,z\in V$ and $1\leq k\leq d_{w}-1,1\leq m\leq d_{z}-1$,%
\begin{equation*}
\lambda _{w}^{(k)}/\lambda _{z}^{(m)}\notin (A(L)\cup A(L^{c}))^{\mathbb{Q}}.
\end{equation*}
\end{enumerate}

\noindent \noindent If $G$ contains a directed circuit not passing through a
vertex $u$, then $F_{u}$ is not the attractor of any COSC standard IFS.
\end{lemma}

\begin{proof}
Let $L$ be a directed circuit that does not go through $u$. By condition
\textrm{(}$ii^{\prime }$\textrm{)} we know that condition $(ii)$ in Lemma %
\ref{main thm} holds upon taking any vertex $v$ in $L$. Since the digraph $G$
is strongly connected, there exists a directed path from $u$ to $v$. Since
condition $(iii)$ remains the same, we only need to verify condition $(i)$
in Lemma \ref{main thm} under the stronger condition \textrm{(}$i^{\prime }$%
\textrm{)}. For, suppose that there exists some $\theta \in (A(L))^{\mathbb{Q%
}}\cap (A(L^{c}))^{\mathbb{Q}_{+}^{\ast }}$. Setting%
\begin{equation*}
A(L)=\{b_{i}\}_{i=1}^{m},\ A(L^{c})=\{c_{j}\}_{j=1}^{n},
\end{equation*}%
so that $A=A(L)\cup A(L^{c})$, we write%
\begin{equation*}
\theta
=\tprod\limits_{i=1}^{m}b_{i}^{p_{i}}=\tprod\limits_{j=1}^{n}c_{j}^{q_{j}}
\end{equation*}%
for some two vectors $(p_{i})_{i=1}^{m}\in \mathbb{Q}^{m}$ and $%
(q_{j})_{j=1}^{n}\in (\mathbb{Q}_{+}^{n})^{\ast }$. Then
\begin{equation*}
1=\tprod\limits_{i=1}^{m}b_{i}^{p_{i}}\tprod\limits_{j=1}^{n}c_{j}^{-q_{j}}%
\in A^{\mathbb{Q}^{\ast }},
\end{equation*}%
where we have used that $A(L)\cap A(L^{c})=\emptyset $ since all the
contraction ratios have different absolute values by condition ($i^{\prime }$%
). However, this contradicts our assumption $1\notin A^{\mathbb{Q}^{\ast }}$%
. Therefore, all conditions in Lemma \ref{main thm} are satisfied, thus
the conclusion of the lemma follows.
\end{proof}

\begin{remark}
Lemma \ref{main} is an extension of \cite[Theorem 6.3]{BooreFal2013}. The
assertion of Lemma \ref{main} is optimal in the sense that the restriction
on the graph `there is a circuit not passing through $u$' cannot be relaxed
(see Theorem \ref{5.1} in the Appendix).
\end{remark}

The following example, with a digraph
that has two vertices with two loops and is strongly connected,
illustrates Lemma \ref{main}.

\begin{example}[Two-vertex digraph]
\label{EX1}\textrm{\textrm{Let $G=(V,E)$ be a strongly connected digraph
where $V=\{1,2\}$, $E=\{e_{1}(1),e_{1}(2),e_{2}(1),e_{2}(2)\}$, so that $%
d_{1}=2$, $d_{2}=2$ , see Figure \ref{ver2}. } }
\begin{figure}[h]
\begin{center}
\includegraphics[width=8cm, height=2.5cm]{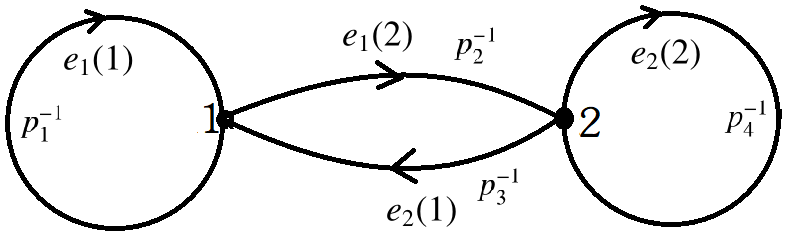}
\end{center}
\caption{$F_1$ and $F_2$ are not the attractors of any standard COSC IFS.}
\label{ver2}
\end{figure}

\textrm{\textrm{Let $\{p_{j}\}_{1\leq j\leq 4}$ be four distinct primes
arranged in ascending order so that $2\leq p_{j}<p_{j+1}$, and let $p_{5}$
be a positive number such that $\log p_{5}$ is not a rational linear
combination of $\{\log p_{j}\}_{1\leq j\leq 4}$. Let $\lambda >0$ be any
real number, and let $x$ be a vector given by
\begin{equation}
x=(p_{1}^{-1},p_{2}^{-1},p_{3}^{-1},p_{4}^{-1},\lambda ,p_{5}\lambda )%
\eqqcolon(x_{1}^{(1)},x_{1}^{(2)},x_{2}^{(1)},x_{2}^{(2)},\xi _{1}^{(1)},\xi
_{2}^{(1)}).  \label{5p}
\end{equation}%
}The matrix $M_{x}(1)$ in (\ref{mx}), (\ref{mij}) associated with point $x$
is given by%
\begin{equation}
M:=M_{x}(1)=\left(
\begin{array}{cc}
x_{1}^{(1)} & x_{1}^{(2)} \\
x_{2}^{(1)} & x_{2}^{(2)}%
\end{array}%
\right) =\left(
\begin{array}{cc}
p_{1}^{-1} & p_{2}^{-1} \\
p_{3}^{-1} & p_{4}^{-1}%
\end{array}%
\right) .  \label{E1}
\end{equation}%
\textrm{Note that $x\in P$ in (\ref{sm}) by using (\ref{large}), since }}%
\begin{equation*}
\mathrm{\mathrm{\max \{p_{1}^{-1}+p_{2}^{-1},p_{3}^{-1}+p_{4}^{-1}\}<1.}}
\end{equation*}%
\textrm{\textrm{Let $\ell =(l_{1},l_{2})^{T}$ be given by (\ref{lu}), that
is,%
\begin{equation}
\ell =\left(
\begin{array}{c}
l_{1} \\
l_{2}%
\end{array}%
\right) =(\mathrm{id}-M_{x}(1))^{-1}\left(
\begin{array}{c}
\xi _{1}^{(1)} \\
\xi _{2}^{(1)}%
\end{array}%
\right) =\left(
\begin{array}{cc}
1-p_{1}^{-1} & -p_{2}^{-1} \\
-p_{3}^{-1} & 1-p_{4}^{-1}%
\end{array}%
\right) ^{-1}\left(
\begin{array}{c}
\lambda \\
p_{5}\lambda%
\end{array}%
\right) .  \label{E2}
\end{equation}%
Let $\mathcal{b}\coloneqq(b_{1}^{(1)},b_{2}^{(1)})$ \ for $%
b_{1}^{(1)},b_{2}^{(1)}\in \mathbb{R}$, and let%
\begin{eqnarray*}
&&b_{1}^{(2)}\coloneqq b_{1}^{(1)}+p_{1}^{-1}l_{1}+\lambda , \\
&&b_{2}^{(2)}\coloneqq b_{2}^{(1)}+p_{3}^{-1}l_{2}+\lambda p_{5}.
\end{eqnarray*}%
We define four similarities $\digamma (x,\mathcal{b}):=%
\{S_{e_{1}(1)},S_{e_{1}(2)},\mathrm{\mathrm{S}}_{e_{2}(1)},S_{e_{2}(2)}\}$,
depending on $x$, $\mathcal{b}$, by%
\begin{eqnarray}
S_{e_{1}(1)}(t) &=&p_{1}^{-1}t+b_{1}^{(1)},\text{ \ }%
S_{e_{1}(2)}(t)=p_{2}^{-1}t+b_{1}^{(2)},  \notag \\
S_{e_{2}(1)}(t) &=&p_{3}^{-1}t+b_{2}^{(1)},\text{ \ }%
S_{e_{2}(2)}(t)=p_{4}^{-1}t+b_{2}^{(2)}\text{ \ for }t\in
%TCIMACRO{\U{211d} }%
%BeginExpansion
\mathbb{R}
%EndExpansion
.  \label{p5}
\end{eqnarray}%
Clearly, such a GD-IFS $\digamma (x,\mathcal{b})$ has absolute contraction
ratio set given by $A:=\{p_{i}^{-1}\}_{i=1}^{4}$. Applying Lemma \ref{P1}, $%
\digamma (x,\mathcal{b})$ satisfies the CSSC, whose basic gap lengths sets
are $\Lambda _{1}=\{\xi _{1}^{(1)}\}=\{\lambda \}$ (at vertex $1$) and $%
\Lambda _{2}=\{\xi _{2}^{(1)}\}=\{\lambda p_{5}\}$ (at vertex $2)$. Let $%
F_{1},F_{2}$ be the attractors of $\digamma (x,\mathcal{b})$ at vertices $1$
and $2$. } }

\textrm{\textrm{We will use Lemma \ref%
{main} to  show that $F_{1}$ (or $F_{2}$) is not the attractor of
any COSC standard IFS, noting that $(V,E)$ contains a directed circuit (loop) not
passing through vertex $1$ (or through vertex $2$).  } }

\textrm{\textrm{Condition $(i^{\prime })$ is clear since the
contraction ratios $A=\{p_{i}^{-1}\}_{i=1}^{4}$ are distinct, and $1\notin
A^{\mathbb{Q}^{\ast }}$ by using Proposition \ref{easy} in the Appendix.
Condition $(ii^{\prime })$ is trivial since the basic gap lengths are $%
\lambda ,\lambda p_{5}$ that are strictly positive. } }

\textrm{\textrm{It remains to verify condition $(iii)$, or equivalently to
check that
\begin{eqnarray*}
\frac{\lambda _{2}^{(1)}}{\lambda _{1}^{(1)}} &=&\frac{\xi _{2}^{(1)}}{\xi
_{1}^{(1)}}=\frac{\lambda p_{5}}{\lambda }=p_{5}\notin A^{\mathbb{Q}}, \\
\frac{\lambda _{1}^{(1)}}{\lambda _{2}^{(1)}} &=&\frac{1}{p_{5}}\notin A^{%
\mathbb{Q}}.
\end{eqnarray*}%
However, this is trivial by noting that
\begin{equation*}
p_{5}\neq \left( p_{1}^{-1}\right) ^{s_{1}}\left( p_{2}^{-1}\right)
^{s_{2}}\left( p_{3}^{-1}\right) ^{s_{3}}\left( p_{4}^{-1}\right) ^{s_{4}}%
\text{ \ (the same is true for }p_{5}^{-1}\text{)}
\end{equation*}%
for any rationals $(s_{i})_{i=1}^{4}$, since $\log p_{5}$ is not a rational
linear combination of $\{\log p_{j}\}_{1\leq j\leq 4}$. } }

\textrm{\textrm{Therefore, all the assumptions $(i^{\prime }),(i^{\prime
\prime }),(iii)$ in Lemma \ref{main} are satisfied, so the GD-attractor $F_{1}$
(or $F_{2}$) is not the attractor of any standard IFS with the COSC. } }
\end{example}

We next show that for $n$-dimensional Lebesgue almost all vectors in $P$,
all the conditions in Lemma \ref{main} hold for their corresponding GD-IFSs.
Let $P_{1}$ be a subset of $P$ given by%
\begin{equation}
P_{1}:=\{x\in P_{0}:r_{\sigma }(M_{x}(1))<1,\xi _{i}^{(k)}>0\text{ for each
vertex }i\in V\text{, }1\leq k\leq d_{i}-1\}.  \label{P1-2}
\end{equation}%
Clearly, $P_{1}\subset P$ since each $\sum_{k=1}^{d_{i}-1}\xi _{i}^{(k)}>0$.

\begin{definition}[Admissible set]
\label{ADS}With the notation as above, we say that a point $%
x=(x_{1},x_{2},\cdots ,x_{n})$ in the set $P_{1}$ is \emph{admissible if}%
\begin{equation}
\tprod_{i=1}^{n}|x_{i}|^{p_{i}}\neq \tprod_{i=1}^{n}|x_{i}|^{q_{i}}
\label{13}
\end{equation}%
for any two distinct vectors $(p_{i})_{i=1}^{n}$ and $(q_{i})_{i=1}^{n}$ of
nonnegative rationals. The set of all admissible points is denoted by $%
\mathcal{A}$.
\end{definition}

Note that the admissible set $\mathcal{A}$\ depends only on the numbers of
vertices and their out-degrees, but is independent of any vertex itself and
the order of edges. If $(x_{1},x_{2},\cdots ,x_{n})\in \mathcal{A}$, then
for any two distinct indices $i,j$, taking $p_{i}=1,p_{k}=0$ for all $k\neq
i $ and $q_{j}=1,q_{k}=0$ for all $k\neq j$ in (\ref{13}),%
\begin{equation}
|x_{i}|\neq |x_{j}|,  \label{72}
\end{equation}%
and so the entries of any vector in $\mathcal{A}$ all have distinct absolute
values.

By Lemma \ref{P1}, we know that each admissible point $x$ gives arise to a
COSC GD-IFS
\begin{equation}
\digamma (x,\mathcal{b})  \label{ff1}
\end{equation}%
in a way of (\ref{S1}), (\ref{lu}), for any $\mathcal{b}$ in (\ref{bb2}),

The following says that the size of the admissible set $\mathcal{A}$ is very
large.

\begin{theorem}
\label{COSC except}Let $G=(V,E)$ be a strongly connected digraph with $%
d_{w}\geq 2$ for all $w\in V$, containing a vertex $u\in V$ outside a
directed circuit. With the notation as above, if $x\in \mathcal{A}$ then
the attractor $F_{u}$ of the
corresponding GD-IFS, $\digamma (x,\mathcal{b})$, defined as in $(\ref{ff1})$
for any $\mathcal{b}$, is not the attractor of any COSC standard IFS.
Moreover, with $n$ given as in $(\ref{nn})$,
\begin{equation}
\mathcal{L}^{n}(P\setminus \mathcal{A})=0,  \label{LL3}
\end{equation}%
that is, the complement of the set $\mathcal{A}$ in $P$ has $n$-dimensional
Lebesgue measure zero.
\end{theorem}

\begin{proof}
Let $\mathcal{b}=(b_{i}^{(1)})_{i\in V}$ for $b_{i}^{(1)}\in \mathbb{R}$ and
let $x=(x_{1},x_{2},\cdots ,x_{n})$ be an admissible point. By Lemma \ref{P1}%
, the corresponding GD-IFS, $\digamma (x,\mathcal{b})$ associated with the
vectors $x,\mathcal{b}$, satisfies the CSSC.\textbf{\ }We will show that
such a GD-IFS $\digamma (x,\mathcal{b})$ also satisfies all three conditions
$(i^{\prime })$, $(ii^{\prime })$, $(iii)$ in Lemma \ref{main}.

Clearly, the GD-IFS $\digamma (x,\mathcal{b})$ satisfies condition $%
(ii^{\prime })$ by noting that $\Lambda _{i}\neq \emptyset $ for each vertex
$i\in V$, since all the basic gap lengths sitting at vertex $i$ are $\xi
_{i}^{(1)},\xi _{i}^{(2)},\cdots ,\xi _{i}^{(d_{i}-1)}$ by Lemma \ref%
{P1}$(ii)$, which are strictly positive since the vector $x$
belongs to $P_{1}$.

We show condition $(i^{\prime })$. Let
\begin{equation*}
X:=\{|x_{i}|\}_{i=1}^{n}\subset (0,\infty ).
\end{equation*}%
We need to prove%
\begin{equation}
1\notin X^{\mathbb{Q}^{\ast }}.  \label{888}
\end{equation}%
For, suppose that $1=\tprod_{i=1}^{n}|x_{i}|^{s_{i}}$ for some $%
(s_{i})_{i=1}^{n}\in (\mathbb{Q}^{n})^{\ast }$, then%
\begin{equation*}
\tprod_{i=1}^{n}|x_{i}|^{s_{i}^{-}}=\tprod_{i=1}^{n}|x_{i}|^{s_{i}^{+}}
\end{equation*}%
where $s_{i}^{+}=\max \{s_{i},0\},\ s_{i}^{-}=\max \{-s_{i},0\}$ so that $%
s_{i}=s_{i}^{+}-s_{i}^{-}$. As not all $s_{i}$ are zero, we see that $%
(s_{i}^{+})_{i=1}^{n}\neq (s_{i}^{-})_{i=1}^{n}$ are two distinct
nonnegative rational vectors. This contradicts the admissibility of $x$ as
defined in (\ref{13}), thus (\ref{888}) is true.

By using (\ref{S1}) and (\ref{72}), all the contraction ratios of the COSC
GD-IFS $\digamma (x,\mathcal{b})$ have different absolute values. Since $%
1\notin A^{\mathbb{Q}^{\ast }}$ as $A^{\mathbb{Q}^{\ast }}\subset X^{\mathbb{%
Q}^{\ast }}$, where $A$ is the absolute contraction ratio set of $\digamma (x,%
\mathcal{b})$, condition $(i^{\prime })$ is satisfied.

For condition $(iii)$, suppose that there exists some$\ a\in A^{\mathbb{Q}}\
$such that $a=\lambda _{w}^{(k)}/\lambda _{z}^{(m)}$, where $\lambda
_{w}^{(k)}\in \Lambda _{w},\ \lambda _{z}^{(m)}\in \Lambda _{z}$. Then
\begin{equation*}
1=\lambda _{w}^{(k)}\left( \lambda _{z}^{(m)}\right) ^{-1}a^{-1}\in
\{|x_{i}|\}^{\mathbb{Q}^{\ast }}=X^{\mathbb{Q}^{\ast }}
\end{equation*}%
by noting that $\lambda _{w}^{(k)}=x_{i}>0$, $\lambda _{z}^{(m)}=x_{j}>0$
for some two indices $i\neq j$ in virtue of definition (\ref{pp}),
contradicting (\ref{888}). Thus condition $(iii)$ is also satisfied.

Therefore, by applying Lemma \ref{main}, the attractor $F_{u}$ of the GD-IFS
$\digamma (x,\mathcal{b})$ is not the attractor of any COSC standard IFS.

We finally show that $\mathcal{L}^{n}(P\setminus \mathcal{A})=0$. For this,
note that
\begin{equation}
\mathcal{L}^{n}(P\setminus P_{1})=0  \label{p-p1}
\end{equation}%
where $P_{1}$ is defined as in (\ref{P1-2}), since $P\setminus P_{1}$ lies
in the union of hyperplanes $\xi _{i}^{(k)}=0$. We just need to show $%
\mathcal{L}^{n}(P_{1}\setminus \mathcal{A})=0$. Let
\begin{equation*}
x=(x_{1},x_{2},\cdots ,x_{n})\in P_{1}\setminus \mathcal{A},
\end{equation*}%
that is, for some two distinct vectors $(p_{i})_{i=1}^{n}$ and $%
(q_{i})_{i=1}^{n}$ of nonnegative rationals,
\begin{equation*}
\tprod_{i=1}^{n}|x_{i}|^{p_{i}}=\tprod_{i=1}^{n}|x_{i}|^{q_{i}}.
\end{equation*}%
As $p_i \neq q_i$ for some $i$, say without loss of generality for $i=1$, then
\begin{equation*}
|x_{1}|=\tprod_{i=2}^{n}|x_{i}|^{(q_{i}-p_{i})/(p_{1}-q_{1})},
\end{equation*}%
from which, it follows that any vector in $P_{1}\setminus \mathcal{A}$ lies
in an at most $(n-1)$-dimensional manifold. Since there are countably many such
equations, the union of countably many such manifolds has $n$-dimensional
Lebesgue measure zero in $\mathbb{R}^{n}$.
\end{proof}

There are a plenty of examples of admissible points so that the assertions
of Theorem \ref{COSC except} hold. However, there are also some other
interesting examples such that the first assertion in Theorem \ref{COSC
except} still holds but points are not admissible.

\begin{example}
\textrm{\label{R2}The point $x$ given by (\ref{5p}) in Example \ref{EX1} is
not admissible in the sense of Definition \ref{ADS} for a certain class of }$%
\lambda $\textrm{. To see this, we need to show that (\ref{13}) fails for
suitable $\lambda $. In fact, if (\ref{13}) fails, then by definition (\ref%
{5p})%
\begin{eqnarray*}
\left( p_{1}^{-1}\right) ^{s_{1}}\left( p_{2}^{-1}\right) ^{s_{2}}\left(
p_{3}^{-1}\right) ^{s_{3}}\left( p_{4}^{-1}\right) ^{s_{4}}\lambda
^{s_{5}}\left( p_{5}\lambda \right) ^{s_{6}}
&=&\tprod_{i=1}^{6}|x_{i}|^{s_{i}}=\tprod_{i=1}^{6}|x_{i}|^{t_{i}} \\
&=&\left( p_{1}^{-1}\right) ^{t_{1}}\left( p_{2}^{-1}\right) ^{t_{2}}\left(
p_{3}^{-1}\right) ^{t_{3}}\left( p_{4}^{-1}\right) ^{t_{4}}\lambda
^{t_{5}}\left( p_{5}\lambda \right) ^{t_{6}}
\end{eqnarray*}%
for some two distinct vectors $(s_{i})_{i=1}^{6}$ and $(t_{i})_{i=1}^{6}$ of
nonnegative rationals. From this, we know that%
\begin{equation}
\lambda
^{(s_{5}-t_{5})+(s_{6}-t_{6})}=p_{1}^{s_{1}-t_{1}}p_{2}^{s_{2}-t_{2}}p_{3}^{s_{3}-t_{3}}p_{4}^{s_{4}-t_{4}}p_{5}^{-(s_{6}-t_{6})}.
\label{p4}
\end{equation}%
Thus, condition (\ref{13}) fails if $\lambda $ is chosen as in (\ref{p4}).
In particular, condition (\ref{13}) fails if $\lambda =\frac{1}{\sqrt{p_{5}}}
$ on taking $s_{i}=t_{i}$ for $i=1,2,3,4$ whilst $s_{i}=t_{i}+1$ for $i=5,6,$
}

\textrm{However, the GD-attractor }$F_{u}$, \textrm{associated with such a
non-admissible point }$x$, \textrm{is not the attractor of any COSC standard
IFS by Example \ref{EX1}. }
\end{example}

We further consider the situation by removing the `COSC'. We will apply
Corollary \ref{C2} and Theorem \ref{App} in the Appendix.

\begin{theorem}
\label{T2}Let $G=(V,E)$ be a strongly connected digraph with $d_{j}\geq 2$
for every vertex $j\in V$, containing a vertex $i\in V$ outside a directed
circuit. Let $x\in \mathcal{A}(\delta )$ (see definition $(\ref%
{AA1})$) satisfying that, for every vertex $j\neq i$ in $V$,%
\begin{eqnarray}
|x_{i}^{(1)}| &\in &\left( |x_{j}^{(1)}|+\cdots
+|x_{j}^{(m_{j})}|+(m_{j}-1)\delta ,|x_{j}^{(1)}|+\cdots
+|x_{j}^{(m_{j})}|+m_{j}\delta \right) ,  \label{X1} \\
1-|x_{i}^{(1)}| &\in &\left( |x_{j}^{(1)}|+\cdots
+|x_{j}^{(n_{j})}|+(n_{j}-1)\delta ,|x_{j}^{(1)}|+\cdots
+|x_{j}^{(n_{j})}|+n_{j}\delta \right) ,  \label{X2}
\end{eqnarray}%
where $m_{j},n_{j}\in [1,d_{j}-1]$ are integers. Let $\digamma
(x) $ be corresponding CSSC GD-IFS constructed as in Corollary \ref{C2},
with GD-attractors $(F_{j})_{j\in V}$. Then $F_{i}$ is not the attractor of
any standard IFS.
\end{theorem}

\begin{proof}
Let $x\in \mathcal{A}(\delta )$. Recall that the corresponding GD-IFS, $%
\digamma (x)=\{S_{e_{i}(k)}\}_{i\in V,1\leq k\leq d_{i}}$ associated with
point $x$, is given by (\ref{S5}), where $\{b_{i}^{(k+1)}\}_{i\in V,1\leq
k\leq d_{i}-1}$ are real numbers in $(0,1)\ $defined as in (\ref{bb5}) ($%
b_{i}^{(1)}=0$ for every $i\in V$).

We apply Theorem \ref{App} in the Appendix to prove this theorem. Clearly,
conditions $(1)$, $(2)$ in Theorem \ref{App} are satisfied. \ In order to
verify condition $(3)$, we need to show that for every
vertex $j\neq i$,
\begin{equation}
F_{i}\nsubseteq F_{j}  \label{S1-3}
\end{equation}%
and
\begin{equation}
1-F_{i}\nsubseteq F_{j}.  \label{S1-4}
\end{equation}%
We first show (\ref{S1-3}). Indeed, note that the point $|x_{i}^{(1)}|$
belongs to the attractor $F_{i}$ by Corollary \ref{C2}$(ii)$. However, this
point does not belong to any attractor $F_{j}$ ($j\neq i$), since it falls
in some basic gap (see formula (\ref{62})) of $F_{j}$ by using assumption (%
\ref{X1}).

Similarly, the point $1-|x_{i}^{(1)}|$ belongs to the set $1-F_{i}$ but does
not belong to any attractor $F_{j}$ ($j\neq i$), since it also falls in some
basic gap of $F_{j}$ by using assumption (\ref{X2}), and thus (\ref{S1-4})
is also true, as required.
\end{proof}

To illustrate Theorem \ref{T2}, we give an example. Let $i$ be a fixed
vertex in$\ V=\{1,2,\cdots ,N\}$. Let $x\in \mathcal{A}%
(\delta )$ satisfy
\begin{equation}
|x_{i}^{(1)}|\in \left( |x_{j}^{(1)}|,|x_{j}^{(1)}|+\delta \right) \text{ \
and \ }1-|x_{i}^{(1)}|\in \left( |x_{j}^{(1)}|,|x_{j}^{(1)}|+\delta \right)
\text{ for any }j\neq i, \label{X3}
\end{equation}%
so that both conditions (\ref{X1}), (\ref{X2}) are satisfied with $%
m_{j}=n_{j}=1$. To secure (\ref{X3}), we let
\begin{equation*}
0<\delta <\min_{j\in V}\left\{ \frac{1}{2(d_{j}-1)}\right\} .
\end{equation*}%
By the definition of $\mathcal{A}(\delta )$, any vector $x\in\mathcal{A}%
(\delta )$ satisfies that for all $j\in V$,
\begin{equation}
|x_{j}^{(1)}|+\cdots +|x_{j}^{(d_{j})}|=1-\left( d_{j}-1\right) \delta >%
\frac{1}{2}\text{.}  \label{mtg}
\end{equation}%
Now we first choose $\{x_{j}^{(1)}\}_{j\in V}$ by
\begin{equation*}
|x_{i}^{(1)}|=\frac{1}{2}\text{ and }|x_{j}^{(1)}|\in \left( \frac{1}{2}%
-\delta ,\frac{1}{2}\right) \text{ for any }j\neq i,
\end{equation*}%
and then we choose $\{x_{j}^{(2)},x_{j}^{(3)},\cdots
,x_{j}^{(d_{j})}\}_{j\in V}$ to be any numbers such that (\ref{mtg}) is
satisfied. Such a class of points satisfy condition (\ref{X3}), which
implies that conditions (\ref{X1}), (\ref{X2}) are both satisfied.

\section{Appendix}

In this appendix we derive some general properties and secondary results
 that are used in the main part of the paper.

The following proposition on ordering integer lattice points is used in
Lemma \ref{GDR}. Recall that $\mathbb{Z}%
_{+}$ denotes the set of all nonnegative integers.

\begin{proposition}
Let $B\subset\mathbb{Z}_{+}^{n}$ be an infinite set. Then $B$ contains two
distinct vectors $\overrightarrow{x}\leq \overrightarrow{y}$ under the
partial order defined by inequality of all coordinates.\label{part}
\end{proposition}

\begin{proof}
We write $\overrightarrow{x}:=(x_{i})_{i=1}^{n}\in \mathbb{Z}_{+}^{n}$.
Consider the set of integers:
$$S := \big\{\min \{x_{i}\}_{i=1}^{n}: \overrightarrow{x}\in B\big\}.$$
If $S$ is unbounded, then we are done by fixing some vector $%
\overrightarrow{x}$ and taking $\overrightarrow{y}=(y_{i})_{i=1}^{n}\in
B\subset \mathbb{Z}_{+}^{n}$ with $\min \{y_{i}\}_{i=1}^{n}>\max
\{x_{i}\}_{i=1}^{n}$ so that
\begin{equation*}
\overrightarrow{x}<\overrightarrow{y}.
\end{equation*}%
Otherwise $S$ is bounded by an integer $N$ in which case we prove the proposition
 by induction on $n$. When $n=1$ it is
trivial. Assume that the proposition holds for $n-1$. For each $1\leq j\leq n$
and each $\alpha \in \{0,1,\cdots ,N\}$, define
\begin{equation*}
B(\alpha ,j):=\{\overrightarrow{x}\in B:x_{j}=\min
\{x_{i}\}_{i=1}^{n}=\alpha \},
\end{equation*}%
a (possibly empty) collection of all vectors in $B$ whose $j$-th entries equal to the same
number $\alpha $ and take the smallest value. Since
\begin{equation*}
\mathop{\textstyle \bigcup }_{j=1}^{n}\mathop{\textstyle \bigcup }_{\alpha
=0}^{N}B(\alpha ,j)=B,
\end{equation*}%
we can assume that some $B(\alpha ,j)$, say $B(\beta ,m)$,
contains infinitely many elements. Deleting the $m$th coordinate $%
x_{m}=\beta $ of all the vectors in such a set $B(\beta ,m)$, we obtain an
infinite set $B^{\prime }(\beta ,m)\subset \mathbb{Z}_{+}^{n-1}$, and by
induction assumption, $B^{\prime }(\beta ,m)\subset \mathbb{Z}_{+}^{n-1}$
has two distinct vectors $\overrightarrow{x^{\prime }}\leq \overrightarrow{%
y^{\prime }}$. Inserting the $m$th coordinate $x_{m}=\beta $ into $%
\overrightarrow{x^{\prime }},\ \overrightarrow{y^{\prime }}$ to get $%
\overrightarrow{x},\ \overrightarrow{y}$ respectively, we obtain two
distinct vectors $\overrightarrow{x}\leq \overrightarrow{y}$ in $B(\beta
,m)\subset B$, showing the assertion for $\mathbb{Z}_{+}^{n}$.
\end{proof}

The next proposition generalises a well-known result for standard IFSs
to GD-IFSs.

\begin{proposition}
\label{APP2}Let $G=(V,E)$ be a digraph and $(F_{u})_{u\in V}$ be the
GD-attractors of a GD-IFS $\digamma =(V,E,(S_{e})_{e\in E})$ based on it. If
there exist non-empty sets $\left( U_{u}\right) _{u\in V}$ such
that
\begin{equation}
\mathop{\textstyle \bigcup }_{v\in V}\mathop{\textstyle \bigcup }_{e\in
E_{uv}}S_{e}\left( U_{v}\right) \subset U_{u}\text{ for each }u\in V
\label{app1}
\end{equation}%
then $F_{u}\subset \overline{U_{u}}$, the closure of the set $U_{u}$, for
each $u\in V$.
\end{proposition}

\begin{proof}
Set $I_{u}:=\overline{U_{u}}$ for each $u\in V$. Let $I_{u}^{m}$ be defined
by (\ref{000}) for $m\geq 1$. Then the inclusion (\ref{001}) is satisfied,
since
\begin{eqnarray*}
I_{u}^{1} &=&\mathop{\textstyle \bigcup }_{e\in E_{u}^{1}}S_{e}\left(
I_{\omega (e)}\right) =\mathop{\textstyle \bigcup }_{e\in
E_{u}^{1}}S_{e}\left( \overline{U_{\omega (e)}}\right) =\mathop{\textstyle
\bigcup }_{e\in E_{u}^{1}}\overline{S_{e}\left( U_{\omega (e)}\right) } \\
&\subset &\overline{\mathop{\textstyle \bigcup }_{e\in E_{u}^{1}}S_{e}\left(
U_{\omega (e)}\right) }\subset \overline{U_{u}}=I_{u}\text{ \ (using (\ref%
{app1})),}
\end{eqnarray*}%
thus showing that $F_{u}\subset I_{u}^{1}\subset \overline{U_{u}}$ by virtue
of (\ref{003}). The proof is complete.
\end{proof}

The directed paths in GD-IFSs play the same role as the finite-length words
in standard IFSs, as the following proposition suggests. We will frequently
use the following fact that, for any $u\in V$ and $m\geq 1$,
\begin{equation}
F_{u}=\mathop{\textstyle \bigcup }_{\mathbf{e}\in E_{u}^{m}}S_{\mathbf{e}%
}(F_{\omega (\mathbf{e})}),  \label{ee1}
\end{equation}%
by repeatedly using definition (\ref{gdattract}) (recall that $E_{u}^{m}$ is
the totality of all paths of length $m$ leaving $u$). The following proposition
concerns the disjointness of images of components under mappings
corresponding to different words.

\begin{proposition}
\label{CSC}Let $G=(V,E)$ be a digraph and $(F_{u})_{u\in V}$ be the
GD-attractors of a GD-IFS $\digamma =(V,E,(S_{e})_{e\in E})$ based on it.
Assume that each $F_{u}$ is not a singleton. Let $\mathbf{e}^{\prime }$, $%
\mathbf{e}^{\prime \prime }$ be two directed paths with $\mathbf{e}^{\prime
\prime }\neq \mathbf{e}^{\prime }\mathbf{e}$ if $\vert
\mathbf{e}^{\prime }|\leq |\mathbf{e}^{\prime \prime }|$ (where $\mathbf{e}$
is a directed path which may be empty). If $\digamma $ satisfies the COSC on
$%
%TCIMACRO{\U{211d} }%
%BeginExpansion
\mathbb{R}
%EndExpansion
$, then the interiors of $S_{\mathbf{e}^{\prime }}(\mathrm{conv}\,F_{\omega (%
\mathbf{e}^{\prime })})$ and $S_{\mathbf{e}^{\prime \prime }}(\mathrm{conv}%
\,F_{\omega (\mathbf{e}^{\prime \prime })})$ are disjoint. Similarly, if $%
\digamma $ satisfies the CSSC then $S_{\mathbf{e}^{\prime }}(F_{\omega (%
\mathbf{e}^{\prime })})$ and $S_{\mathbf{e}^{\prime \prime }}(F_{\omega (%
\mathbf{e}^{\prime \prime })})$ are disjoint.
\end{proposition}

\begin{proof}
By (\ref{ee1}), for any path $\mathbf{e}$, we have $S_{\mathbf{e}}(F_{\omega
(\mathbf{e})})\subset F_{\alpha (\mathbf{e})}$, where $\alpha (\mathbf{e})$
denotes the initial vertex of path $\mathbf{e}$. As $S_{%
\mathbf{e}}$ is a similarity on $%
%TCIMACRO{\U{211d} }%
%BeginExpansion
\mathbb{R}
%EndExpansion
$, taking the convex hulls gives that
\begin{equation*}
S_{\mathbf{e}}(\mathrm{conv}\,F_{\omega (\mathbf{e})})=\mathrm{conv}\text{ }%
S_{\mathbf{e}}(F_{\omega (\mathbf{e})})\subset \mathrm{conv}\,F_{\alpha (%
\mathbf{e})},
\end{equation*}%
from which, we see that, for any path $\mathbf{e}_{1}\mathbf{e}_{2}$
(meaning that $\omega (\mathbf{e}_{1})=\alpha (\mathbf{e}_{2})$, the
terminal of $\mathbf{e}_{1}$ is the initial of $\mathbf{e}_{2}$),
\begin{equation}
S_{\mathbf{e}_{1}\mathbf{e}_{2}}(\mathrm{conv}\,F_{\omega (\mathbf{e}%
_{2})})=S_{\mathbf{e}_{1}}(S_{\mathbf{e}_{2}}(\mathrm{conv}\,F_{\omega (%
\mathbf{e}_{2})}))\subset S_{\mathbf{e}_{1}}(\mathrm{conv}\,F_{\alpha (%
\mathbf{e}_{2})})=S_{\mathbf{e}_{1}}(\mathrm{conv}\,F_{\omega (\mathbf{e}%
_{1})}).  \label{con}
\end{equation}

Assume now that $\digamma $ satisfies the COSC. By (\ref{con3}), one can take%
\begin{equation*}
U_{u}=\mathrm{int}(\mathrm{conv}\,F_{u})\text{ \ for each }u\in V,
\end{equation*}%
which is non-empty by our assumption that $F_{u}$ is not a singleton.

For any two paths $\mathbf{e}e_{1}$, $\mathbf{e}e_{2}$ with common path $%
\mathbf{e}$ and distinct edges $e_{1},e_{2}$, the interiors of two intervals%
\begin{equation}
\text{int(}S_{\mathbf{e}e_{1}}(\mathrm{conv}\,F_{\omega (e_{1})})\text{)}%
\cap \text{int(}S_{\mathbf{e}e_{2}}(\mathrm{conv}\,F_{\omega (e_{2})})\text{)%
}=\emptyset  \label{con2}
\end{equation}%
by using the COSC, since
\begin{equation*}
S_{\mathbf{e}e_{1}}(\mathrm{conv}\,F_{\omega (e_{1})})=S_{\mathbf{e}%
}(S_{e_{1}}(\mathrm{conv}\,F_{\omega (e_{1})}))\text{ \ and \ }S_{\mathbf{e}%
e_{2}}(\mathrm{conv}\,F_{\omega (e_{2})})=S_{\mathbf{e}}(S_{e_{2}}(\mathrm{%
conv}\,F_{\omega (e_{2})}))
\end{equation*}%
and the interiors of $S_{e_{1}}(\mathrm{conv}\,F_{\omega (e_{1})})$ and $%
S_{e_{2}}(\mathrm{conv}\,F_{\omega (e_{2})})$ are disjoint as the edges $%
e_{1},e_{2}$ have the same initial vertex, namely the terminal of path $\mathbf{e}$.

Let $\mathbf{e}$ be the longest common path of $\mathbf{e}^{\prime \prime }$
and $\mathbf{e}^{\prime }$ (which may be empty). Write $\mathbf{e}^{\prime }=%
\mathbf{e}e_{1}\mathbf{p}_{1}$ and $\mathbf{e}^{\prime \prime }=\mathbf{e}%
e_{2}\mathbf{p}_{2}$, where $e_{1}\neq e_{2}$ are two distinct edges and $%
\mathbf{p}_{1}$, $\mathbf{p}_{2}$ are some paths (possibly empty). By (\ref%
{con}),
\begin{eqnarray*}
S_{\mathbf{e}^{\prime }}(\mathrm{conv}\,F_{\omega (\mathbf{e}^{\prime })})
&=&S_{\mathbf{e}e_{1}\mathbf{p}_{1}}(\mathrm{conv}\,F_{\omega (\mathbf{p}%
_{1})})\subset S_{\mathbf{e}e_{1}}(\mathrm{conv}\,F_{\omega (e_{1})}), \\
S_{\mathbf{e}^{\prime \prime }}(\mathrm{conv}\,F_{\omega (\mathbf{e}^{\prime
\prime })}) &=&S_{\mathbf{e}e_{2}\mathbf{p}_{2}}(\mathrm{conv}\,F_{\omega (%
\mathbf{p}_{2})})\subset S_{\mathbf{e}e_{2}}(\mathrm{conv}\,F_{\omega
(e_{2})}),
\end{eqnarray*}%
thus the interiors of $S_{\mathbf{e}^{\prime }}(\mathrm{conv}\,F_{\omega (%
\mathbf{e}^{\prime })})$ and $S_{\mathbf{e}^{\prime \prime }}(\mathrm{conv}%
\,F_{\omega (\mathbf{e}^{\prime \prime })})$ are disjoint by using (\ref%
{con2}).

The assertion for the CSSC is similar. The proof is complete.
\end{proof}

The following was essentially proved in \cite[Lemma 5.1]{Boore2014}\textbf{,
}except that we also consider the COSC case.

\begin{theorem}
\label{5.1} Let $G=(V,E)$ be a strongly connected digraph with $d_{v}\geq 2$
for all $v\in V$. If every directed circuit goes through a vertex $u\in V$,
then for any (resp. COSC) GD-IFS based on $G$, its attractor $F_{u}$ is also
the attractor of a (resp. COSC) standard IFS.
\end{theorem}

\begin{proof}
Set $N:=\#V$, the number of vertices in $V$. Let $L(u)$ be the set of all
circuits having $u$ as their initial and terminal, and which do not contain
another shorter circuits, that is,%
\begin{equation*}
L(u):=\{\mathbf{e}=e_{uv_{1}v_{2}\cdots v_{k}u}:\text{ each }v_{i}\neq u\in
V,|\mathbf{e}|\leq N\},
\end{equation*}%
where the symbol $e_{uv_{1}v_{2}\cdots
v_{k}u}=e_{uv_{1}}e_{v_{1}v_{2}}\cdots e_{v_{k}u}$ is understood to be a
path consisting of consecutive edges.

We claim that
\begin{equation}
F_{u}=\mathop{\textstyle \bigcup }_{\mathbf{e}^{\prime }\in L(u)}S_{\mathbf{e%
}^{\prime }}(F_{u}),  \label{ee}
\end{equation}%
by using the fact that every circuit goes through vertex $u$.

To see this, we have by (\ref{ee1}) that
\begin{equation*}
F_{u}=\mathop{\textstyle \bigcup }_{\mathbf{e}\in E_{u}^{N}}S_{\mathbf{e}%
}(F_{\omega (\mathbf{e})}).
\end{equation*}%
Since any directed path $\mathbf{e}$ in $E_{u}^{N}$ can be written as%
\begin{equation*}
\mathbf{e}=e_{uv_{1}v_{2}\cdots v_{N}},
\end{equation*}%
we see that at least one of vertices $v_{1},v_{2},\cdots ,v_{N}$ must be $u$%
, otherwise, one of them would appear twice, thus producing a circuit,
contradicting the assumption that every directed circuit goes through vertex $%
u$. There exists some index $k$ such that $v_{k}=u$ and the path visits $u$
the second time (besides the initial time), and
\begin{equation*}
\mathbf{e}=e_{uv_{1}v_{2}\cdots v_{k-1}uv_{k+1}\cdots
v_{N}}=e_{uv_{1}v_{2}\cdots v_{k-1}u}e_{uv_{k+1}\cdots v_{N}}=\mathbf{e}%
^{\prime }\mathbf{e}^{\prime \prime },
\end{equation*}%
where $\mathbf{e}^{\prime }=e_{uv_{1}v_{2}\cdots v_{k-1}u}\in L(u)$ and $%
\mathbf{e}^{\prime \prime }$ is a path with initial $u$ if it exists
(possibly $\mathbf{e}^{\prime \prime }$ is empty and the following argument
will become easier). From this, we know that%
\begin{equation*}
S_{\mathbf{e}}(F_{\omega (\mathbf{e})})=S_{\mathbf{e}^{\prime }}(S_{\mathbf{e%
}^{\prime \prime }}(F_{\omega (\mathbf{e}^{\prime \prime })}))\subset S_{%
\mathbf{e}^{\prime }}(F_{u}),
\end{equation*}%
since $S_{\mathbf{e}^{\prime \prime }}(F_{\omega (\mathbf{e}^{\prime \prime
})})\subset F_{u}$ by (\ref{ee1}). It follows that
\begin{equation*}
F_{u}=\mathop{\textstyle \bigcup }_{\mathbf{e}\in E_{u}^{N}}S_{\mathbf{e}%
}(F_{\omega (\mathbf{e})})\subset \mathop{\textstyle \bigcup }_{\mathbf{e}%
^{\prime }\in L(u)}S_{\mathbf{e}^{\prime }}(F_{u}).
\end{equation*}

The opposite inclusion is also clear since, by (\ref{ee1}),
\begin{equation*}
S_{\mathbf{e}^{\prime }}(F_{u})\subset F_{u},
\end{equation*}%
thus showing that our claim (\ref{ee}) holds true. Therefore, $F_{u}$ is the
attractor of the IFS
\begin{equation}
\Phi :=\{S_{\mathbf{e}^{\prime }}:\mathbf{e}^{\prime }\in L(u)\}.  \label{hi}
\end{equation}

If the GD-IFS $\digamma $ further satisfies the COSC, we claim that the
IFS $\Phi $ given by (\ref{hi}) also satisfies the COSC. Indeed, by
definition of the COSC and the fact that $\Phi $ has attractor $F_{u}$, we
need only to show that the interiors of two intervals $S_{\mathbf{e}^{\prime
}}(\mathrm{conv}\,F_{u})$ and $S_{\mathbf{e}^{\prime \prime }}(\mathrm{conv}%
\,F_{u})$ are disjoint, where $\mathbf{e}^{\prime },\mathbf{e}^{\prime
\prime }$ are in $L(u)$. But this assertion immediately follows from
Proposition \ref{CSC}.
\end{proof}

The following easy property of powers of primes is used in the examples
in Section \ref{Sect4}.

\begin{proposition}
\label{easy}Let $\{a_{i}\}_{i=1}^{n}$ be distinct positive prime numbers.
Then
\begin{equation*}
1\notin A^{\mathbb{Q}^{\ast }}\text{ \ for }A=\{a_{i}^{-1}\}_{i=1}^{n}.
\end{equation*}
\end{proposition}

\begin{proof}
Suppose to the contrary, that $1\in A^{\mathbb{Q}^{\ast }}$. Then $%
1=\tprod_{i=1}^{n}\left( a_{i}^{-1}\right) ^{s_{i}}$ for some non-zero
vector $(s_{i})_{i=1}^{n}$\ of rationals. Let $q$ be the least common denominator
of the rationals $s_{i}$. Taking the $q$th power, it follows that%
\begin{equation*}
m:=\tprod_{i=1}^{n}a_{i}^{qs_{i}^{-}}=\tprod_{i=1}^{n}a_{i}^{qs_{i}^{+}}
\end{equation*}%
where $s_{i}^{+}=\max \{s_{i},0\},\ s_{i}^{-}=\max \{-s_{i},0\}$ so that $%
s_{i}=s_{i}^{+}-s_{i}^{-}$. As the $s_{i}$ are not all zero, the vectors of
integers $(qs_{i}^{+})_{i=1}^{n}$ and $(qs_{i}^{-})_{i=1}^{n}$ are distinct.
By the uniqueness of the prime factorisation of the integer $m$, we see that
$(qs_{i}^{+})_{i=1}^{n}=(qs_{i}^{-})_{i=1}^{n}$, a contradiction.
\end{proof}

The following assertion was essentially obtained in \cite[Theorem 1.4 and
the end of Section $1$]{Boore2014}.\ Here we give a simpler proof under
stronger assumptions withconditions $(2)$, $(3)$ in the next
theorem.

\begin{theorem}
\label{App} Let $G=(V,E)$ be a strongly connected digraph with $d_{w}\geq 2$
for each $w\in V$. Suppose that a given GD-IFS of similarities based on $G$
satisfies the CSSC, and $\mathrm{conv}\,F_{w}=[0,1]$ for each $w\in V$. For
some vertex $u\in V$, suppose the following conditions hold.

\begin{enumerate}
\item[$(1)$] There is a directed circuit that does not pass through $u$.

\item[$(2)$] All basic gaps have the same length $\delta >0$.

\item[$(3)$] For each vertex $v\neq u$, we have $F_{u}\nsubseteq F_{v}$ and $%
1-F_{u}\nsubseteq F_{v}$.
\end{enumerate}

Then $F_{u}$ is not the attractor of any standard IFS defined on $\mathbb{R}$%
.
\end{theorem}

\begin{proof}
The proof is divided into two steps.

\emph{Step 1.} We claim that, for any $v\in V$ and any contracting similarity
$f$ with $f(F_{u})\subset F_{v}$, there exists some path $\mathbf{e}$
leaving $v$ with terminal $\omega (\mathbf{e})=u$ such that
\begin{equation}
f(F_{u})\subset S_{\mathbf{e}}(F_{u}).  \label{Step1}
\end{equation}%
Indeed, as $F_{v}$ consists of the level-$1$ cells $S_{e}(F_{\omega (e)})$
for edges $e$ leaving $v$ by using (\ref{gdattract}), the $f(F_{u})$ must
belong to only one of those cells, say%
\begin{equation}
f(F_{u})\subset S_{e}(F_{\omega (e)})\text{ \ for some edge }e\text{ leaving
}v.  \label{S3}
\end{equation}%
Otherwise, there are two points in $f(F_{u})$ lying in two distinct level-$%
1 $ cells, and as $f(F_{u})\subset F_{v}$, we know that $f(F_{u})$ spans a
basic gap of $F_{v}$, implying that $f(F_{u})$ has a gap, containing a basic
gap of $F_{v}$, whose length is clearly greater than or equal to $\delta $.
However, this is impossible, because all gap lengths of $F_{u}$ do not
exceed $\delta $ by assumption $(2)$ and  (\ref{gls}), so that all the gap
lengths of $f(F_{u})$ are strictly smaller than $\delta $ by using the
contractivity of $f$.

By (\ref{ee1}), it follows that%
\begin{equation*}
f(F_{u})\subset F_{v}=\mathop{\textstyle \bigcup }\limits_{\mathbf{e}%
^{\prime }\in E_{v}^{m}}S_{\mathbf{e}^{\prime }}(F_{\omega (\mathbf{e}%
^{\prime })})\text{ \ \ for any }m\geq 1,
\end{equation*}%
where $E_{v}^{m}$ is the set of all paths leaving $v$ with the same length $%
m $ as before. As $f(F_{u})$ has fixed diameter and cells $S_{\mathbf{e}%
^{\prime }}(F_{\omega (\mathbf{e}^{\prime })})$ have arbitrarily small diameters
by taking $m$ large, we can choose a longest directed path $\mathbf{e}%
_{1}$ leaving $v$, which exists by using (\ref{S3}) and the fact that
distinct cells of the same length are disjoint (see Proposition \ref{CSC}),
such that
\begin{equation}
f(F_{u})\subset S_{\mathbf{e}_{1}}(F_{\omega (\mathbf{e}_{1})})\text{.}
\label{eee}
\end{equation}

We show that the contraction ratio $\rho $ of the mapping $S_{\mathbf{e}%
_{1}}^{-1}\circ f$ satisfies $\rho =\pm 1$.

The diameter of each $F_{w}$ equals $1$, since $\mathrm{%
conv}\,F_{w}=[0,1]$ for each $w\in V$ by our assumption. By (\ref{eee})%
\begin{equation}
S_{\mathbf{e}_{1}}^{-1}\circ f(F_{u})\subset F_{\omega (\mathbf{e}_{1})},
\label{S4}
\end{equation}%
implying that $|\rho |\leq 1$ by comparing the diameters of $%
F_{u}$ and $F_{\omega (\mathbf{e}_{1})}$ and noting that $S_{\mathbf{e}%
_{1}}^{-1}\circ f$ is a similarity.

If $|\rho |<1$, we will derive a contradiction. Indeed, by (\ref{S4}), we
can apply (\ref{S3}), with $f$ being replaced by $S_{\mathbf{e}%
_{1}}^{-1}\circ f$ and $v$ replaced by $\omega (\mathbf{e}_{1})$, and obtain%
\begin{equation*}
S_{\mathbf{e}_{1}}^{-1}\circ f(F_{u})\subset S_{e}(F_{\omega (e)})
\end{equation*}%
for some edge $e$ leaving $\omega (\mathbf{e}_{1})$. From this and
that $\omega (e)=\omega (\mathbf{e}_{1}e)$,%
\begin{equation*}
f(F_{u})\subset S_{\mathbf{e}_{1}}(S_{e}(F_{\omega (e)}))=S_{\mathbf{e}%
_{1}e}(F_{\omega (\mathbf{e}_{1}e)}),
\end{equation*}%
which contradicts the fact that $\mathbf{e}_{1}$ is the longest path by
virtue of (\ref{eee}). Thus $|\rho |=1$.

Therefore, if $\rho =1$, then $F_{u}+c\subset F_{\omega (\mathbf{e}_{1})}$
for some translation $c\in
%TCIMACRO{\U{211d} }%
%BeginExpansion
\mathbb{R}
%EndExpansion
$ using (\ref{S4}), which implies that%
\begin{equation*}
\lbrack 0,1]+c=(\mathrm{conv}\,F_{u})+c=\mathrm{conv}\left( F_{u}+c\right)
\subset \mathrm{conv}\,F_{\omega (\mathbf{e}_{1})}=[0,1]
\end{equation*}%
using our assumption that $\mathrm{conv}\,F_{w}=[0,1]$ for each $w\in V$.
Then $c=0$, and so
\begin{equation*}
F_{u}\subset F_{\omega (\mathbf{e}_{1})},
\end{equation*}%
showing that $\omega (\mathbf{e}_{1})=u$ by assumption $(3)$ that $%
F_{u}\nsubseteq F_{v}$ if $v\neq u$.

Similarly, if $\rho =-1$, then $-F_{u}+c\subset F_{\omega (\mathbf{e}_{1})}$
for some translation $c\in
%TCIMACRO{\U{211d} }%
%BeginExpansion
\mathbb{R}
%EndExpansion
$ by (\ref{S4}), which implies that%
\begin{equation*}
\lbrack -1,0]+c=\mathrm{conv}\left( -F_{u}\right) +c=\mathrm{conv}\left(
-F_{u}+c\right) \subset \mathrm{conv}\,(F_{\omega (\mathbf{e}_{1})})=[0,1]
\end{equation*}%
using our assumption that $\mathrm{conv}\,F_{w}=[0,1]$ for each $w\in V$.
We must have $c=1$, and so
\begin{equation*}
1-F_{u}\subset F_{\omega (\mathbf{e}_{1})},
\end{equation*}%
showing that $\omega (\mathbf{e}_{1})=u$ again by assumption $(3)$ that $%
1-F_{u}\nsubseteq F_{v}$ if $v\neq u$.

Therefore, noting that $\omega (\mathbf{e}_{1})=u$ in (\ref{eee}), we obtain
(\ref{Step1}) with $\mathbf{e}=\mathbf{e}_{1}$, proving our claim.

\emph{Step 2.} We show that $F_{u}$ is not the attractor of any standard IFS.

Assume to the contrary that there exists a standard IFS $\{f_{i}\}$ such
that%
\begin{equation*}
F_{u}=\mathop{\textstyle \bigcup }_{i}f_{i}(F_{u}).
\end{equation*}%
As $f_{i}(F_{u})\subset F_{u}$, using (\ref{Step1}) with $v=u$, we know that
$f_{i}(F_{u})\subset S_{\mathbf{e}_{i}}(F_{u})$, and so
\begin{equation}
F_{u}=\mathop{\textstyle \bigcup }_{i}f_{i}(F_{u})\subset \mathop{\textstyle
\bigcup }_{i}S_{\mathbf{e}_{i}}(F_{u}),  \label{ss}
\end{equation}%
where each $\mathbf{e}_{i}$ is a directed circuit from initial $u$ to
terminal $u$. By condition $(1)$, there is a vertex $w\neq u$ contained in a
circuit $L$ that does not pass through $u$. By the strong connectivity, we
can find a simple path $L_{1}$ (i.e. a path visits any vertex
at most for once) from $u$ to $w$.

Note that the path $L_{1}L^{m}$  from $u$ to $w$ visits $u$ only once.
We can pick an integer $m$ so large that the path length is greater than $\max \{|%
\mathbf{e}_{i}|\}_{i}$. By (\ref{ee1}) and (\ref{ss}),%
\begin{equation}
S_{L_{1}L^{m}}(F_{w})\subset F_{u}\subset \mathop{\textstyle \bigcup }_{i}S_{%
\mathbf{e}_{i}}(F_{u}).  \label{L8}
\end{equation}

Note that $\{S_{e}\}$ satisfies the CSSC by assumption $(2)$, and so $%
S_{L_{1}L^{m}}(F_{w})$ is disjoint with any set $S_{\mathbf{e}_{i}}(F_{u})$
in (\ref{L8})  using Proposition \ref{CSC}, since the path $L_{1}L^{m}$
does not start with any of these paths $\mathbf{e}_{i}$, otherwise $%
L_{1}L^{m}$ would visit $u$ twice. This contradicts (%
\ref{L8}), thus showing that $F_{u}$ is not self-similar.
\end{proof}

\end{document}